\documentclass[11pt]{amsart}
\usepackage{geometry}                % See geometry.pdf to learn the layout options. There are lots.
\usepackage{amssymb}
\usepackage{amsthm}
\usepackage{amsmath}
\usepackage{bm}
\usepackage{graphicx}
\usepackage{epstopdf}
\usepackage{subfig}
\usepackage{pdfpages}
\usepackage{centernot}
\usepackage{color}
\usepackage{url}
\usepackage{float}

\usepackage{mathptmx}      % use Times fonts if available on your TeX system
\usepackage[ruled,vlined]{algorithm2e}

\SetKwFunction{defined}{defined}
\SetKwFunction{notdefined}{not defined}
\SetKwFunction{mean}{mean}
\SetKwProg{Fn}{Function}{}{}
\SetKwFunction{FnGetLowerBound}{GetAlphaLowerBound}%
\SetKwFunction{FnGetAlpha}{GetAlphaEstimate}%

\newcommand{\mb}{\mathbf}

\newcommand{\boldeta}{\boldsymbol\eta}
\newcommand{\Wx}{W_{\mb x}}
\newcommand{\Wb}{W_{\mb b}}
\newcommand{\xp}{{\mb x}_0}
\newcommand{\btrue}{{\mb b}_\text{true}}
\newcommand{\xtrue}{{\mb x}_\text{true}}
\newcommand{\Cb}{C_{\mb b}}
\newcommand{\Cm}{C_{\mb m}}
\newcommand{\Cx}{C_{\mb x}}

\newcommand{\ck}{k}
\newcommand{\kopt}{k_{\mathrm{opt}}}
\newcommand{\kmax}{k_{\mathrm{max}}}
\newcommand{\alphamax}{\alpha_{\mathrm{max}}}
\newcommand{\alphamin}{\alpha_{\mathrm{min}}}

\newcommand{\galpha}{\alpha}

\newtheorem{theorem}{Theorem}[section]
\newtheorem{prop}{Proposition}[section]

\newtheorem{assumption*}{Assumption}

\newtheorem{remark}{Remark}[section]
\newtheorem{corollary}{Corollary}[section]

\newcommand{\argmin}[1]{\textnormal{arg} \min_{#1}}
\begin{document}
\title{Unbiased Predictive Risk Estimation of the Tikhonov Regularization Parameter. \thanks{Rosemary Renaut acknowledges the support of NSF grant DMS 1418377: ``Novel Regularization for Joint Inversion of Nonlinear Problems".}}
%\subtitle{Convergence with increasing rank approximations of the Singular Value Decomposition}
\author{Rosemary A. Renaut, Anthony W. Helmstetter and Saeed Vatankhah}
%\institute{R. Renaut \at School of Mathematical and Statistical Sciences \\Arizona State University \\Tempe\\AZ 85287-1804\\ \email{renaut@asu.edu} \and A. Helmstetter \at 
%School of Mathematical and Statistical Sciences \\Arizona State University \\Tempe\\AZ 85287-1804\\ \email{Anthony.Helmstetter@asu.edu} \and S. Vatankhah \at  Institute
%of Geophysics \\ University of Tehran\\ Iran\\ \email{svatan@ut.ac.ir}}
\date{\today}
\begin{abstract}The truncated singular value decomposition may be used to find the solution of   linear discrete ill-posed problems  in conjunction with Tikhonov regularization and requires the estimation of a regularization parameter that balances between the sizes of the fit to data function   and the regularization term.  The unbiased predictive risk estimator is one suggested method for finding the regularization parameter  when  the noise  in the measurements is normally distributed  with known variance.  In this paper we provide an algorithm using the unbiased predictive risk estimator that automatically finds both the regularization parameter and the number of terms to use from the singular value decomposition. Underlying the algorithm is a new result that proves that   the regularization parameter  converges  with the number of terms from the singular value decomposition. For the analysis it is sufficient to assume that the discrete Picard condition is satisfied for exact data and that noise completely contaminates the measured data coefficients for a sufficiently large number of terms, dependent on both the noise level and the degree of ill-posedness of the system. A lower bound for the regularization parameter is provided leading to a computationally efficient algorithm. Supporting results are compared with those obtained using the method of generalized cross validation.  Simulations for two-dimensional examples verify the theoretical analysis and the effectiveness of the algorithm for increasing noise levels, and demonstrate that the relative reconstruction errors obtained using the truncated singular value decomposition are less than those obtained using the singular value decomposition. \\
This is a pre-print of an article published in BIT Numerical Mathematics. The final authenticated version is available online at: \url{https://doi.org/10.1007%2Fs10543-019-00762-7}

\keywords{Inverse Problems \and Tikhonov Regularization \and unbiased predictive risk estimation \and regularization parameter}
%\subclass{65F10}
\end{abstract}
\maketitle
\pagestyle{headings}
\markright{Convergence with increasing rank approximations of the Singular Value Decomposition}\markleft{Convergence with increasing rank approximations of the Singular Value Decomposition}

\section{Introduction}\label{sec:intro}
We consider the   solution of $A\mb x \approx \mb b$, or $A\mb x \approx \mb \btrue +\boldeta=\mb b$ for noise (measurement error) $\boldeta$, where $A \in \mathcal{R}^{m\times n}$ is ill-conditioned, and the system of equations arises from the discretization of an ill-posed inverse problem that may be over or under determined. The general Tikhonov regularized linear least squares problem
\begin{align}
\mb x^* &= \argmin{\mb x} \lbrace \|A\mb x - \mb b\|_{\Wb}^2 + \|D(\mb x-\xp)\|^2_{\Wx}\rbrace,\label{Jls}
\end{align}
is a well-accepted approach for finding a smooth solution $\mb x$. Here  $\xp$ is given prior information, possibly the mean of $\mb x$,  $\Wb$ and $\Wx$ are weighting matrices on the data fidelity and  regularization terms,  resp., and $D$ is an optional regularization operator.  Often $D$ is imposed as a spatial differential operator,  controlling  the size of the derivative(s) of $\mb x$, but then~\eqref{Jls} can be brought into standard form in which $D$ is replaced by $I$, \cite{bjorck,Hansen:98}. Further, \eqref{Jls}  can be rewritten in terms of a new variable $\mb y = \mb x -\xp$. The weighted norm is defined by $\|\mb x\|^2_{W} : = \mb x^TW\mb x$ and  we use the notation   $\mb m \sim \mathcal{N}(\bm m_0, \Cm)$ for random vector $\mb m$ normally distributed with expected value $E(\bm m)=\bm m_0$ and covariance matrix $\Cm$;  $E(\cdot)$ is used to denote expected value. When  $\boldeta \sim\mathcal{N}(0,\Cb)$, then     $\Wb=\Cb^{-1}$  whitens the noise, i.e.   $\Wb^{1/2}\boldeta \sim\mathcal{N}(0,I)$. Matrix $\Wx=\Cx^{-1}$ can serve similarly as a prior on the inverse covariance of the noise in $D \mb y $. Using $\Wx=\alpha^2I$, as will be assumed here,  corresponds to assuming the posterior distribution $D \mb y \sim\mathcal{N}(0, \alpha^{-2})$, see e.g. \cite{mere:09}. Here we   discuss the solution of~\eqref{Jls} with $\xp=0$, $D=I$,  $\Wx=\alpha^2 I$,   $\Wb=I$ and explicitly assume common variance, $\sigma^2$, in the noise, $\boldeta\sim\mathcal{N}(0,\sigma^2 I)$.

While solutions of \eqref{Jls}  have been extensively studied, e.g. \cite{Hanke,Hansen:98,hansenbook,Vogel:2002} there is still much discussion concerning the selection of $\Wx$ even for the single parameter case,  $\Wx=\alpha^2 I$. Suggested techniques include, among others,  using the Morozov discrepancy principle (MDP) which assumes that the solution should be found within some prescribed $\chi^2$ noise estimate \cite{morozov},  balance of the terms in \eqref{Jls} using the L-curve \cite{Hansen:98}, the quasi-optimality condition \cite{Quasi,doi:10.1080/01630560903392941,HAMARIK20122146} and minimization of the generalized cross validation (GCV) function \cite{GoHeWa} or of the statistically motivated Unbiased Predictive Risk Estimator (UPRE) \cite{stein1981,Vogel:2002}. Of these the MDP, GCV and UPRE approaches are all \textit{a posteriori} estimators, the MDP on the $\chi^2$ distribution of the predicted residual, the GCV through its derivation as a leave one out procedure to minimize the predictive error and the UPRE as an estimator of the minimum predictive risk of the solution. There is an extensive discussion of these methods in the standard literature e.g. \cite{Hansen:98,hansenbook,Vogel:2002} and many more are compared in \cite{BAUER20111795}. We do not replicate that discussion here,   rather we    focus on the UPRE parameter choice method. The UPRE method has a firm theoretical foundation, is robust, and has been extensively applied in practical applications, \cite{Abascal:2008aa,HHSR,lin-2010-upre,mere:09,TomaSixouPeyrin:15,VAR:2014b,VRA:2017,saeed6,saeed7}.  Our analysis extends the approach in \cite{FenuGCV} which  provided bounds on the regularization parameter for  finding $\alpha$ using the GCV; the analysis in \cite{Renaut2017}  that examined convergence of the parameter with increasing resolution of the problem via the connection of the continuous and discrete singular value expansions for specific square integrable operators defining $A$; and the discussion in \cite{RVA:15} that demonstrated the relationship of the  regularization parameter obtained when using the LSQR Krylov method for large scale problems. Moreover, our interest in the UPRE, instead of the MDP,  arises because the UPRE depends only on the underlying knowledge of the noise distribution, whereas the MDP also introduces a secondary tolerance factor on the satisfaction of the $\chi^2$ distribution, which is often needed to limit over smoothing of the solutions, \cite{ABT:13}. 

Throughout we use the Singular Value Decomposition (SVD) $A=U\Sigma V^T$, \cite{GoLo:96}, with columns $\mb u_i$ and $\mb v_i$ of orthonormal $U$ and $V$ respectively, and where the singular values $\sigma_i$  of $A$ are ordered on the principal diagonal of $\Sigma$, from largest to smallest. We assume that the matrix $A$ has effective numerical rank $r$; $\sigma_r>0$, and $\sigma_i$, $i>r$ is effectively zero as determined by the machine precision.   In terms of the SVD components, the solution of \eqref{Jls} is given by
\begin{equation}\label{svdsoln}
\mb x^* = \sum_{i=1}^r \frac{\sigma_i^2}{\sigma_i^2 +\alpha^2} \frac{\mb u_i^T \mb b}{\sigma_i} \mb v_i = \sum_{i=1}^r  \gamma_i(\alpha) \frac{\mb u_i^T \mb b }{\sigma_i} \mb v_i, \quad \gamma_i(\alpha)=\frac{\sigma_i^2}{(\sigma_i^2 +\alpha^2)}.
\end{equation}
The  filter functions are $\gamma_i(\alpha)$ and the given expansion applies, replacing $r$ by $k$, when $A$ is approximated by the TSVD, $A_k = U_k \Sigma_k V_k^T$. Throughout we use the subscript $k$ to indicate variables associated with this rank $k$ approximation, for example regularization parameter $\alpha_k$ indicates the regularization parameter used for the $k$-term TSVD.  Further, the use of the SVD for $A$ provides useful insights on how the UPRE, and other methods, can be implemented when solving \eqref{Jls}. 
Here we will show that the minimization of the underlying UPRE function is efficient and robust with respect to the $k-$term truncated singular value decomposition (TSVD).  Moreover,  there is a resurgence of interest in using a TSVD solution for the solution of ill-posed problems due to the increased feasibility of finding a good approximation of a dominant singular subspace even for large scale problems by using techniques from randomization, e.g. \cite{RandNLA,FastLS,GowerandRichtarik,Mahoney,LSRN,Overdetermined}. Thus the presented results  will be more broadly relevant for efficient estimates of an approximate TSVD using these modern techniques applied for large scale problems, for which it is not feasible to find the full SVD expansion; necessarily $k<<r$.

\textit{Overview of main contributions.}
An open source algorithm, Algorithm~\ref{influx}, for efficiently estimating optimal regularization parameters $\kopt$ and $ \alpha_{\kopt}$, defined to be the optimal number of terms to use from the TSVD, and the associated regularization parameter, resp., is presented. By \texttt{optimal} we mean that these parameters are optimal in the sense of minimizing the UPRE function.  A MATLAB implementation of Algorithm~\ref{influx}  and a $2D$ test case using IR Tools \cite{GaHaNa:IR} is available at \url{ https://github.com/renautra/TSVD_UPRE_Parameter_Estimation}.   A Python 3.* implementation using NumPy and SciPy is also available and relies on provision of the singular values and coefficients $\mb u_i^T \mb b$.  In both cases an estimate for the noise variance in the data is required, as is standard for the UPRE method. The motivation for Algorithm~\ref{influx} is based on the theoretical results presented in  Section~\ref{sec:theory}. 
These results employ standard assumptions on the degree of ill-posedness of the underlying model and on the noise level in the data, \cite{hofmann1986regularization}. We briefly review how both the degree of ill-posedness and the noise level impact the choice of regularization parameter $k$, and demonstrate  that the noise level is far more restrictive so that in general $k\ll r$. %These results indicate why problems that are severely ill-posed   require $k\ll r$ for reasonable levels of noise. 
The convergence of $\alpha_k$, when found using both UPRE and GCV methods, is  illustrated for examples  from the Regularization toolbox  \cite{Regtools}. The theory presented  in Section~\ref{sec:theory} then leads to Theorems~\ref{thm:minbnd} and \ref{thm:conv} which prove a lower bound for $\alpha_k$ and that $\alpha_k$ converges to $  \alpha_{\kopt}$, under the assumption of a unique minimum of the UPRE function. Presented results for image deblurring  verify the practicality of Algorithm~\ref{influx} and demonstrate that the solutions obtained with $\kopt<r$ yield smaller overall relative error than the solutions obtained without truncation of the SVD and $\alpha_r$ found using the UPRE method.

The paper is organized as follows: In Section~\ref{motivation} we present background motivating results based on assumptions on the degree of ill-posedness of the problem in Section~\ref{illposed}, a discussion of numerical rank in Section~\ref{sec:rank},  how noise enters into the problem in  Section~\ref{noiseentering} and   the estimation of the regularization parameter in Section~\ref{curves}. The theoretical results providing our main contributions are presented in Section~\ref{sec:theory}. A practical algorithm for estimating $ \alpha_{\kopt}$, and hence also $\kopt$, is presented in Section~\ref{sec:simulations} with simulations  verifying the analysis and the algorithm for two dimensional cases. Conclusions and future extensions are provided in Section~\ref{sec:conclusions}.

\section{Motivating Results\label{motivation}}
\subsection{Degree of Ill-Posedness\label{illposed}}
As in \cite[Definition 2.42]{hofmann1986regularization}, and subsequently adopted in \cite{Hansen:98},  for the analysis   we assume specific decay rates for the singular values  dependent on whether the problem is mildly, moderately or severely ill-posed. Suppose that $\zeta$ is an arbitrary constant, then   the decay rates are given by
\begin{equation}\label{eq:decay}
\sigma_i = \left\{ \begin{array}{lll}  \zeta i^{-\tau} & \frac12 \le \tau \le 1 & \text{mild ill conditioning}\\ \zeta i^{-\tau}& \tau> 1 & \text{moderate ill conditioning,} \\ \zeta \tau^{-i} &\tau>1 &\text{severe ill conditioning.}\end{array} \right.
\end{equation}
Here $\tau$ is a problem dependent parameter and it is assumed that the decay rates hold on average for sufficiently large $i$. Moreover, while defining $\sigma_i$ in terms of index $i$, as is consistent with the literature, it will also be convenient to consider the definition in terms of the continuous variable $i$, so that $\sigma_{i+\delta}$ is defined also for non-integer $i+\delta$. 
For ease, and without loss of generality,  we pick the constant $\zeta$ in \eqref{eq:decay} so that $\sigma_1=1$ in all cases.  Equivalently we use
\begin{equation}\label{eq:normalizeddecay}
\sigma_i = \left\{ \begin{array}{lll}   i^{-\tau} & \frac12 \le \tau \le 1 & \text{mild ill conditioning,}\\  i^{-\tau}& \tau> 1 & \text{moderate ill conditioning,} \\ \tau^{1-i} &\tau>1 &\text{severe ill conditioning,}\end{array}  \right.
\end{equation}
and note the recurrences
\begin{equation*}%\label{eq:decayrecurrence}
\sigma_{\ell+1} =  \sigma_{\ell} \left\{ \begin{array}{ll}     (\frac{\ell}{\ell+1})^{\tau}  & \text{mild or moderate ill conditioning,}\\ \tau^{-1}  &\text{severe ill conditioning.}\end{array}  \right.
\end{equation*}
\subsection{Numerical Rank}\label{sec:rank}
The precision of the calculations, as determined by the machine epsilon $\epsilon$,  is relevant in terms of the number of  singular values that are significant in the calculation. This is dependent on the decay rate parameters of the singular values. We define the effective rank by $r=\mathrm{argmax} \{i: \sigma_i> \epsilon \sigma_1\}$.
\begin{prop}\label{proprank}
Assuming the normalization of the singular values as given by \eqref{eq:normalizeddecay},  the effective numerical rank $r$ is bounded by
\begin{eqnarray} \label{decaycalc}
  r<\left\{\begin{array}{ll} \epsilon^{-1/ \tau} &\text{mild  / moderate decay, }     \\1-\frac{\log \epsilon}{\log \tau} & \text{severe decay,}\end{array}\right.
  \end{eqnarray}
  where $\epsilon$ is the machine epsilon.
\end{prop}
\begin{proof}
Using \eqref{eq:normalizeddecay}  and normalization $\sigma_1=1$, it is immediate
that  we obtain  \eqref{decaycalc} from
\begin{eqnarray*}\begin{array}{llll}
\text{mild  / moderate:} &r^{-\tau}> \epsilon \text{ implies } r< \epsilon^{-1/ \tau} \\
\text{severe :} &\tau^{1-r}> \epsilon \text{ implies } r< 1-\frac{\log \epsilon}{\log \tau}. \\
\end{array}
\end{eqnarray*}
\end{proof}
Estimates for  numerical rank dependent on the decay rates, are given in Table~\ref{svddecay} for moderate and severe decay. It is immediate that $r$ is very small for cases of severe decay. Hence, for any problem exhibiting this severe decay  and assuming that the discretization is sufficiently fine such that $n\ge r$,  Table~\ref{svddecay} suggests the maximum number of terms that one would  use for the TSVD. Note that apart from the condition $n\ge r$ the results in Table~\ref{svddecay} are effectively independent of the discretization, thus the number of terms that can be used practically is largely independent of the discretization of the problem once $n\ge r$.  Equivalently, with estimates of $\tau$ and $\epsilon$ one may use \eqref{decaycalc} to determine first a minimum $n$ and second the maximum number of terms for the TSVD, the maximum effective numerical rank of the problem.
\begin{table}
\caption{Number of significant singular values $r$ for precision $\epsilon=10^{-15}$ as a function of $\tau$. i.e. $r$ is the numerical rank of the problem.}\label{svddecay}
\begin{tabular}{cc|ccccccccccc}
&$\tau$&$1.25$& $1.50$&$1.75$&$2.00$&$2.50$&$3.00$&$4.00$&$5.00$&$6.00$\\ \hline
Moderate&$r$& $ {1\mathrm{e}{+12}}$&   $ {1\mathrm{e}{+10}}$&   $ 4\mathrm{e}{+8}$&   $3\mathrm{e}{+7}$& ${1\mathrm{e}{+6}}$&     $ 1\mathrm{e}{+5} $ &$5623$&  $1000$  &$316$  \\  \hline
Severe&$r$& $155$&   $ 86$&   $ 62$&   $ 50$& $    38$&     $ 32 $ &$25$&  $22$  &$20$  \\  \hline
\end{tabular}
\end{table}

These results are further illustrated in Figure~\ref{fig:svs} in which we plot the singular values of test problems from the Regularization toolbox, \cite{Regtools}, with the normalization $\sigma_1=1$ in each case. The plots show that these standard one dimensional test cases are primarily severely ill-posed, and thus, according to Table~\ref{svddecay}, guaranteed to have numerically very few accurate terms in the TSVD used for the solution \eqref{svdsoln}. To show the relative independence of $n$ we show in Figure~\ref{fig:svs256} the singular value distributions for the same cases  and on the same scales as in Figure~\ref{fig:svs} but using $n=256$. This verifies  that there is little to be gained by the use of  problems with severe decay, as presented in  \cite{Renaut2017}, to validate convergence of techniques with increasing problem size. The dominant features are always represented by very few terms of the TSVD for cases with severe decay rates of the singular values.

As a comparison we also show in Figure~\ref{fig:exactsvs} the singular values generated using \eqref{decaycalc} with a selection of decay rates, as indicated in the legend. These show the dependence on $\tau$ for mild, moderate and severe decays. Taken together the examples in  Figure~\ref{fig:svmodels} show that the results are not just an artificial artifact of the seemingly strong assumption  in \eqref{eq:normalizeddecay} that $\zeta$ is fixed.

\begin{figure}[!htb]
\begin{center}
\subfloat[Toolbox $n=128$ \cite{Regtools} \label{fig:svs}]{\includegraphics[width=.33\textwidth]{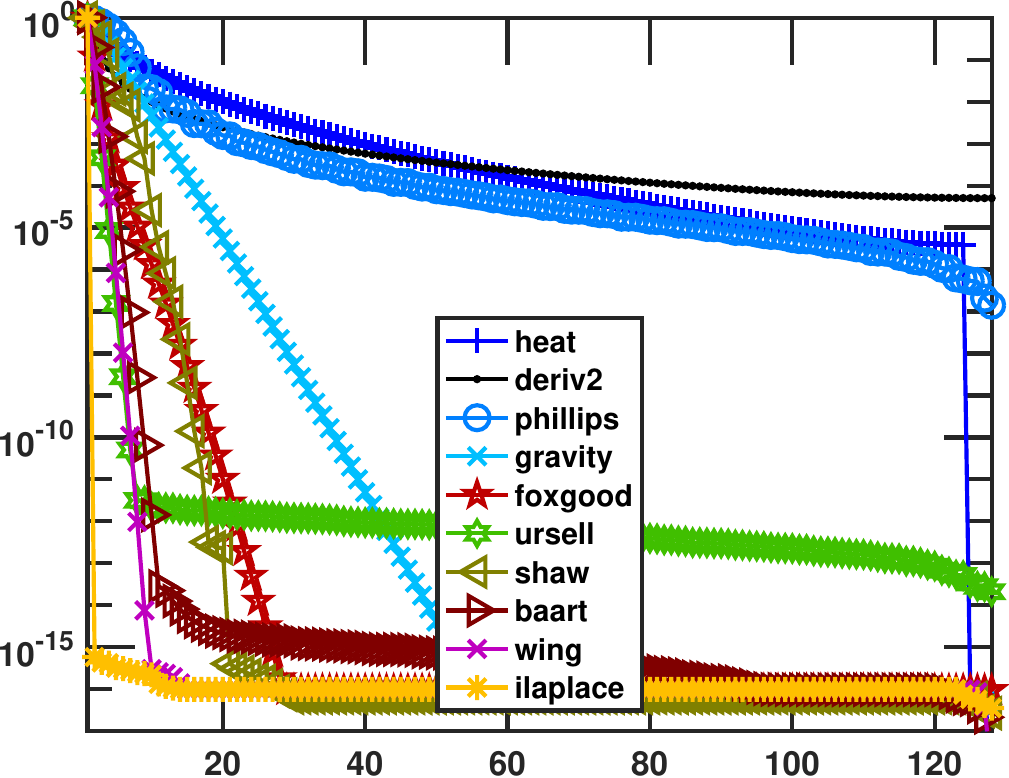}}
\subfloat[Toolbox $n=256$ \cite{Regtools} \label{fig:svs256}]{\includegraphics[width=.33\textwidth]{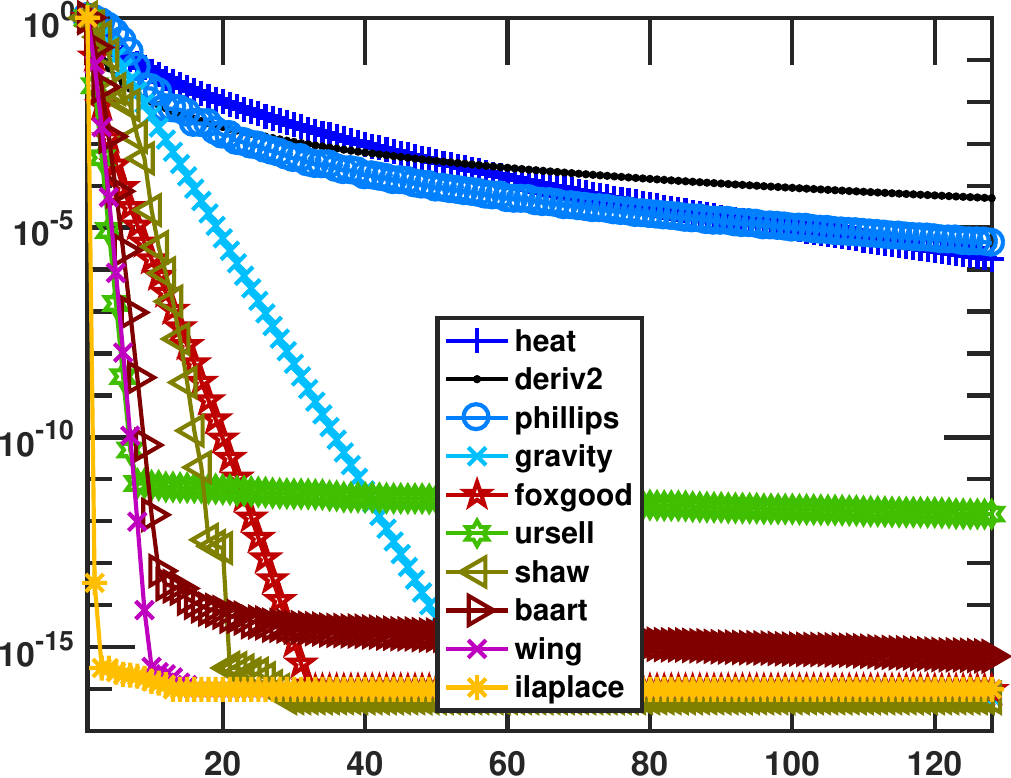}}
\subfloat[Modeled decay rates\label{fig:exactsvs}]{\includegraphics[width=.33\textwidth]{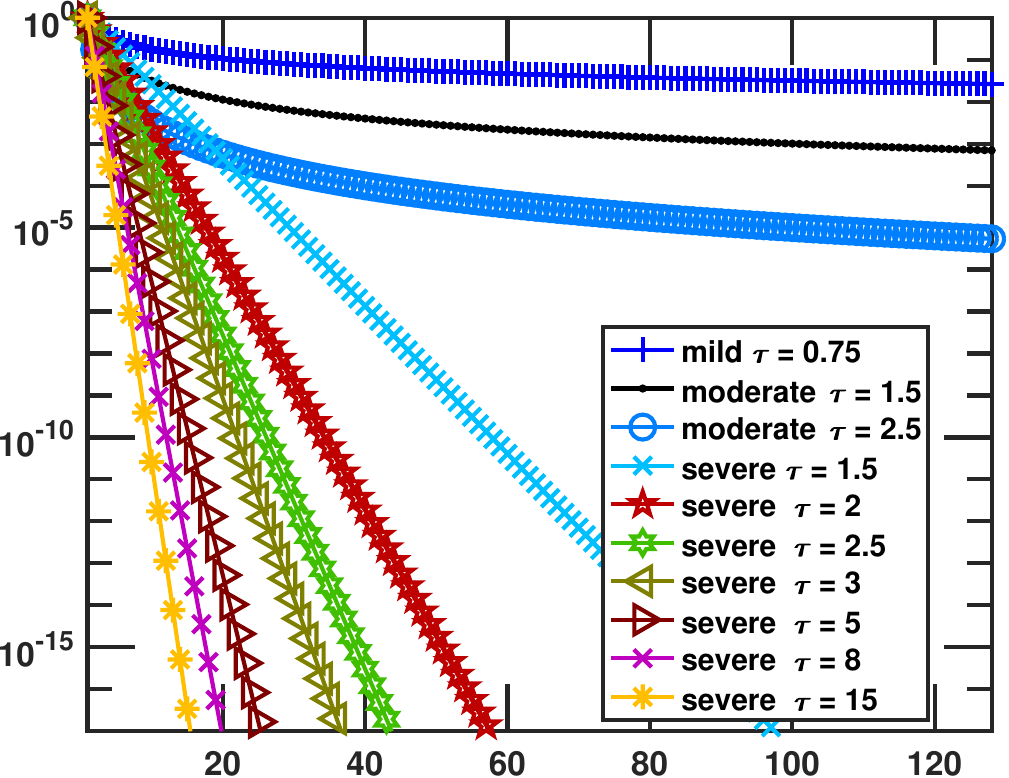}}
\end{center}
\caption{ The singular value distribution for $n=128$ for noted examples from \cite{Regtools}, normalized to $\sigma_1=1$ in Figure~\ref{fig:svs}.  To show the independence of $n$, decay for  $n=256$ for the same examples from \cite{Regtools} is also shown in Figure~\ref{fig:svs256}. For each of the toolbox examples it is possible to  compare with a simulated case for a specific decay rate by illustrating \eqref{eq:normalizeddecay} for severe, moderate or mild decay choices of $\tau$, as appropriate, given in Figure~\ref{fig:exactsvs}\label{fig:svmodels}.}
\end{figure}

\subsection{The Discrete Picard Condition and Noise Contamination}\label{noiseentering}
We now turn to the consideration of the noise in the coefficients $s_i=\mb u_i^T \mb b$ and the impact of this noise  on the potential resolution in the solution, as also discussed in \cite[\S4.8.1]{Hansen:98}. 
First, we assume that the singular values satisfy a decay rate condition  \eqref{eq:normalizeddecay}. We also assume that that the absolute values of the exact coefficients decay at least faster than the singular values \begin{equation}\label{Picard}
(s^2_i)_{\text{true}}\le \sigma_i^{2(1+\nu)} \quad \text{for} \quad 0<\nu<1.
\end{equation}
Then the discrete Picard condition is satisfied, \cite{Hansen1990}, \cite[Theorem 4.5.1]{Hansen:98}. 

When the noise in the data has common variance $\sigma^2$ we assume that there exists $\ell$ such that $E(s_i^2)=\sigma^2$,  for all $i>\ell$. Equivalently, we say that the coefficients are noise dominated for $i>\ell$.  If this does not occur, then either the noise is insignificant, $\sigma^2<\sigma_i^2$ for all $i$, or totally dominates the solution $\sigma^2>1$, and these two cases are not of interest. Thus we can explicitly assume that there exists $\ell$ such that $\sigma_{\ell+1}<\sigma<\sigma_\ell$, and more precisely that 
\begin{eqnarray*}%\label{noiseineq}
\sigma_{\ell+1}^2<\sigma^{1+\nu}_{\ell+1}<\sigma<\sigma^{1+\nu}_{\ell}<\sigma_{\ell},
\end{eqnarray*}
where we use definition \eqref{eq:normalizeddecay}  as a continuous function of $i$ for a non integer index $\ell+\delta$. 
\begin{prop}%\label{prop:noise}
Let   $\sigma=\sigma_{\ell+\delta}^{1+\nu}$ for $0\le \delta<1$ and $0<\nu<1$, then  $E(s_i^2)=\sigma^2$ for $i> \ell$ where
\begin{equation}\label{noiseest}\begin{array}{ll}
\ell \approx \left\{ \begin{array}{ll}    \sigma^{-1/ (\tau(1+\nu))}-\delta& \text{mild  / moderate decay, }     \\
(1-\delta) -\frac{\log \sigma}{(\nu+1)\log \tau}.& \text{severe decay.}\end{array}\right.
\end{array}
\end{equation}
\end{prop}
\begin{proof}
As in the proof of Proposition~\ref{proprank} we solve for $\ell$ dependent on the decay rate with respect to the upper bound in \eqref{noiseest}. This gives
\begin{equation*}
\begin{array}{llll}
\text{mild  / moderate:} &(\ell+\delta)^{-\tau(1+\nu)}= \sigma \text{ implies } \ell= \sigma^{-1/ (\tau(1+\nu))}-\delta, \\
\text{severe :} &\tau^{(1-(\ell+\delta))(1+\nu)}= \sigma \text{ implies } \ell= (1-\delta) -\frac{\log \sigma}{(\nu+1)\log \tau}.
\end{array}
\end{equation*}
\end{proof}

Estimates using $\delta=\nu=0.5$ are indicated in Table~\ref{noiseterms} showing that the number of terms is relatively small even for moderate decay of the singular values for acceptable noise estimates $\sigma$.  Contrasting with Table~\ref{svddecay} we see that the number of  coefficients that can be distinguished from the noise is generally less than the numerical rank of the problem for relevant noise levels and machine precision. This limits the number of the terms of the TSVD to use. In particular, suppose that  $\alpha$ has to be found  to filter the dominant noise terms with index $i\ge \ell$, then  coefficients with $i\gg\ell$ will be further damped because the filter factors given in \eqref{svdsoln} decrease as a function of $i$. These terms then become insignificant in terms of the expansion for the solution.
\begin{table}
\caption{For different noise levels $\sigma$ the size of $\ell$ for given $\tau$ and with $\delta=\nu=0.5$. Entries calculated with rounding using \eqref{noiseest}.}\label{noiseterms}
\begin{tabular}{c|ccccccccccc}\hline
$\tau$&$1.25$& $1.50$&$1.75$&$2.00$&$2.50$&$3.00$&$4.00$&$5.00$&$6.00$\\ \hline
$\sigma$&\multicolumn{10}{c}{Moderate decay} \\ \hline
$1\mathrm{e}{-1}$&     $3$&$2$&$2$&$2$&$1$&$1$&$1$&$1$&1&\\ \hline
 ${1\mathrm{e}{-2}}$&    $11$    & $7$   &  $5$    & $4$    & $3$   &  $2$    & $2$  &   $1$   &  $1$ \\ \hline
 $1\mathrm{e}{-4}$&   $135$  &  $59$ &   $33$&    $21$&    $11$ &    $7$ &    $4$&    $ 3$ &    $2$\\ \hline
 $1\mathrm{e}{-8}$&   $18478$  &  $ 3593  $ &   $ 1115 $&    $464  $&    $135$ &    $59$ &    $21$&    $ 11$ &    $7$\\ \hline
$\sigma$&\multicolumn{10}{c}{Severe decay} \\ \hline
 $1\mathrm{e}{-1}$& $7$  &   $4$ &    $3$ &    $3$&     $2$   & $ 2$     &$2$    & $1$  &   $1$   \\ \hline
 $1\mathrm{e}{-2}$& $14$  &   $8$ &    $6$ &    $5$&     $4$   & $ 3$     &$3$    & $2$  &   $2$   \\ \hline
 $1\mathrm{e}{-4}$& $28$  &   $16$ &    $11$ &    $9$&     $7$   & $ 6$     &$5$    & $4$  &   $4$   \\ \hline
 $1\mathrm{e}{-8}$& $56$  &   $31$ &   $22$& $18$ &    $14$&     $12$   & $ 9$     &$8$    & $7$   \\ \hline
\end{tabular}
\end{table}

\subsection{Regularization Parameter Estimation}\label{curves}
We  deduce from Tables~\ref{svddecay} and \ref{noiseterms} that the number of terms of the TSVD used for the solution of the regularized problem may strongly influence the choice for $\alpha$. Specifically the number of terms  $k$ of the TSVD to use should be less than the numerical rank, $k<r$, and is dependent on the noise level in the data. We are interested in   investigating the choice of $\alpha$ when obtained using the UPRE, but for comparison we also give the GCV function needed for the simulations, and note again that bounds on $\alpha$ dependent on $k$ have already been provided in \cite{FenuGCV}. The GCV and UPRE methods are derived without the use of the SVD, \cite{GoHeWa} and \cite{Vogel:2002}, resp., but it is convenient for the analysis to express both methods in terms of the SVD.  Ignoring constant terms in the UPRE that do not impact the location of the minimum,  introducing $\phi_i(\alpha)=1-\gamma_i(\alpha)=\alpha^2/(\sigma_i^2+\alpha^2)$,  and noting $\gamma_i(\alpha)=0$, for $i>k$, these are given by
\begin{eqnarray}\label{uprefunc}
U_{\ck}(\galpha)&=&\sum_{i=1}^{\ck} (1-\gamma_i(\galpha))^2 ({\mb  u}_i^T {\mb b})^2+2 \sigma^2 \sum_{i=1}^{\ck}\gamma_i(\galpha) = \sum_{i=1}^{\ck} \phi_i^2(\alpha) s_i^2+2 \sigma^2 \sum_{i=1}^{\ck}\gamma_i(\galpha),\\ \label{gcvfunction}
 G_{\ck}(\galpha)&=& \frac{\sum_{i=1}^{m} (1-\gamma_i(\galpha))^2 ({\mb u}_i^T {\mb b})^2}{\left(\sum_{i=1}^{m} (1-\gamma_i(\galpha))\right)^2}= \frac{\sum_{i=1}^{k} \phi^2_i(\galpha)s_i^2+\sum_{i=k+1}^{m} s_i^2}{\left((m-k)+\sum_{i=1}^{k} \phi_i(\galpha)\right)^2}.
\end{eqnarray}
Here the subscript $\ck\le r$ indicates that these are the expressions obtained using the TSVD,
see e.g. \cite[Appendix B]{RVA:15} for derivations of the UPRE and GCV functions for arbitrary pairs $(m,n)$.  Replacing $k$ by $r$ gives the standard functions for the full SVD.  Further, we do not need all terms of the SVD to calculate the numerator in \eqref{gcvfunction}. Using $\|\mb b\|_2^2 = \|U^T \mb b\|_2^2$  we  can use
\begin{eqnarray*}
\sum_{i=k+1}^{m}  ({\mb u}_i^T {\mb b})^2 &=& \|\mb b\|_2^2 - \sum_{i=1}^{k}  ({\mb u}_i^T {\mb b})^2 = \|\mb b\|_2^2 - \sum_{i=1}^{k}  s_i^2.
\end{eqnarray*}

To illustrate how $\alpha_k$ varies  when found using these functions we illustrate an example of a problem that is only moderately  ill-posed ($\tau\approx 1.5$),  showing the results of calculating the UPRE and GCV functions for  data with noise variance $\sigma^2\approx 1\mathrm{e}{-4}$  and $\sigma^2\approx 1\mathrm{e}{-2}$ for the problem \texttt{deriv2}. The  data and solution $\xtrue$ are initially normalized so that $\|\btrue \|_2=1$.  Consistent with the decay rate assumptions the  singular values are normalized by $\sigma_1$. This requires additional normalization of  $\btrue$ by $\sigma_1$, so that eventually $\|\btrue\|_2=\sigma_1^{-1}$.  Then noise contaminated data are generated as $\mb b=\btrue +\boldeta$ for $\boldeta\sim \mathcal{N}(0,\sigma^2I)$, for noise level $\sigma$.   In these examples, the optimal value $\alpha_k$ is obtained by first evaluating $f(\alpha)$, $f(\alpha)=U_k(\alpha)$ or $f(\alpha)=G_k(\alpha)$ as specified in \eqref{uprefunc} or \eqref{gcvfunction}, resp., at $\sigma_i$, $1\le i \le k$. This provides $\alpha_{\mathrm{est}}=\argmin{1\le i\le k} f(\sigma_i)$.  This estimate of the minimum is used as the initial value for minimizing $f(\alpha)$ using Matlab \texttt{fminbnd}  within the interval $[.01 \alpha_{\mathrm{est}}, 100\alpha_{\mathrm{est}}]$. While this choice of lower and upper bounds on $\alpha_k$ is somewhat arbitrary, it is similar to the approach used in \cite{Regtools} for minimizing the UPRE and GCV functions, and is chosen to assure that values for $\alpha_k$ outside the interval $[\sigma_k,\sigma_1]$ are possible when either $\alpha_{\mathrm{est}}=\sigma_k$ or $\sigma_1$.  We pick this specific example in Figure~\ref{fig:functionsderiv2} to highlight the discussion as applied to a problem which is not severely ill-posed. We also give the same information in Figure~\ref{fig:functionsgravity} for the severely ill-posed problem \texttt{gravity} ($\tau\approx 1.5$), see Figure~\ref{fig:svs}.  To gain further insight the Picard plot, plots of $\sigma_i$, $|\mb u_i^T \mb b |$ and the ratio $|\mb u_i^T \mb b |/\sigma_i$, is given in each case in Figures~\ref{fig:lowpicardderiv2} and \ref{fig:highpicardderiv2} for \texttt{deriv2} and in Figures~\ref{fig:lowpicardgravity} and \ref{fig:highpicardgravity} for \texttt{gravity}. The solutions are contaminated by noise very quickly for small $k$, corresponding to fast convergence of $\{\alpha_k\}$ with $k$.

\begin{figure}[!htb]
\begin{center}
\subfloat[$\alpha_k$ convergence. $\sigma^2\approx 1\mathrm{e}{-4}$. \label{fig:alphalowderiv2}]{\includegraphics[width=.25\textwidth]{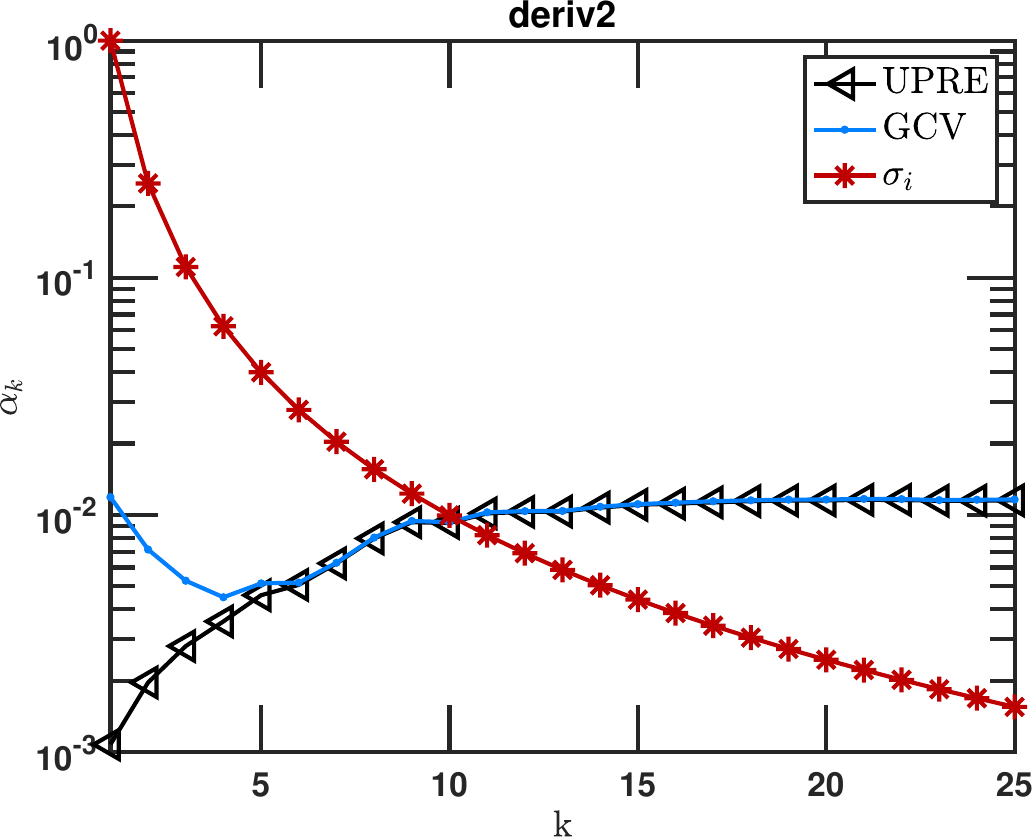}}
\subfloat[Picard Plot. $\sigma^2\approx 1\mathrm{e}{-4}$.\label{fig:lowpicardderiv2}]{\includegraphics[width=.24\textwidth]{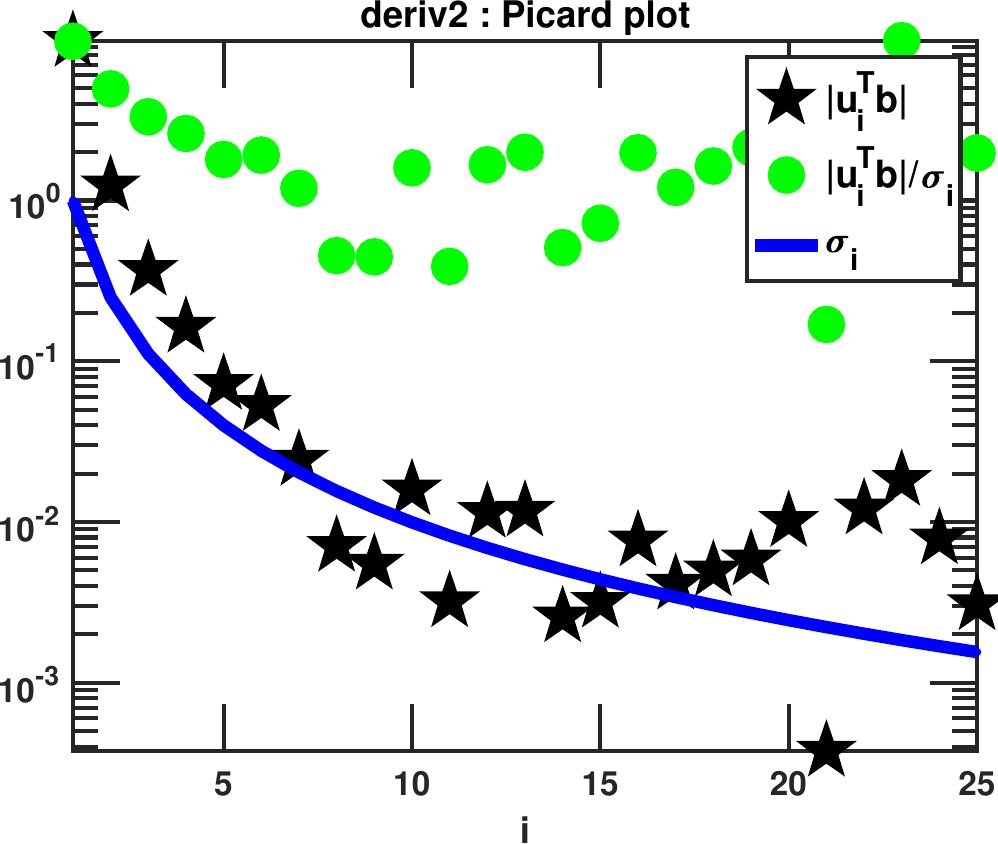}}
\subfloat[$\alpha_k$ convergence.  $\sigma^2\approx 1\mathrm{e}{-2}$.\label{fig:alphahighderiv2}]{\includegraphics[width=.25\textwidth]{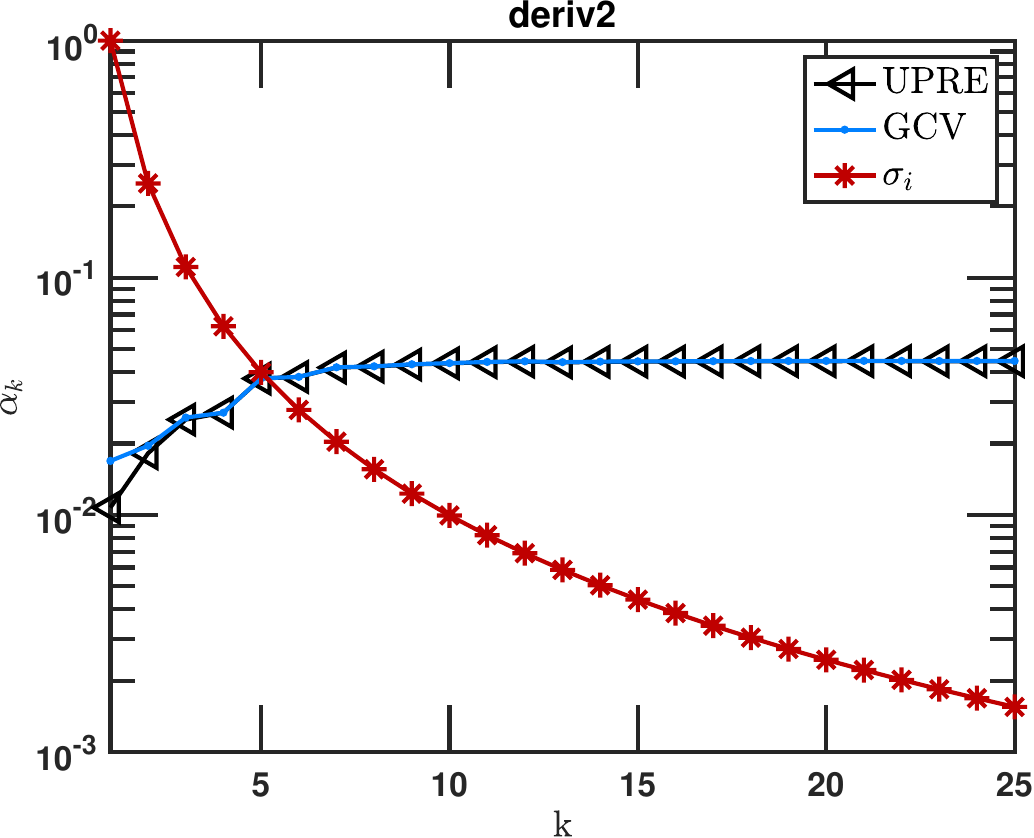}}
\subfloat[Picard Plot. $\sigma^2\approx 1\mathrm{e}{-2}$. \label{fig:highpicardderiv2}]{\includegraphics[width=.24\textwidth]{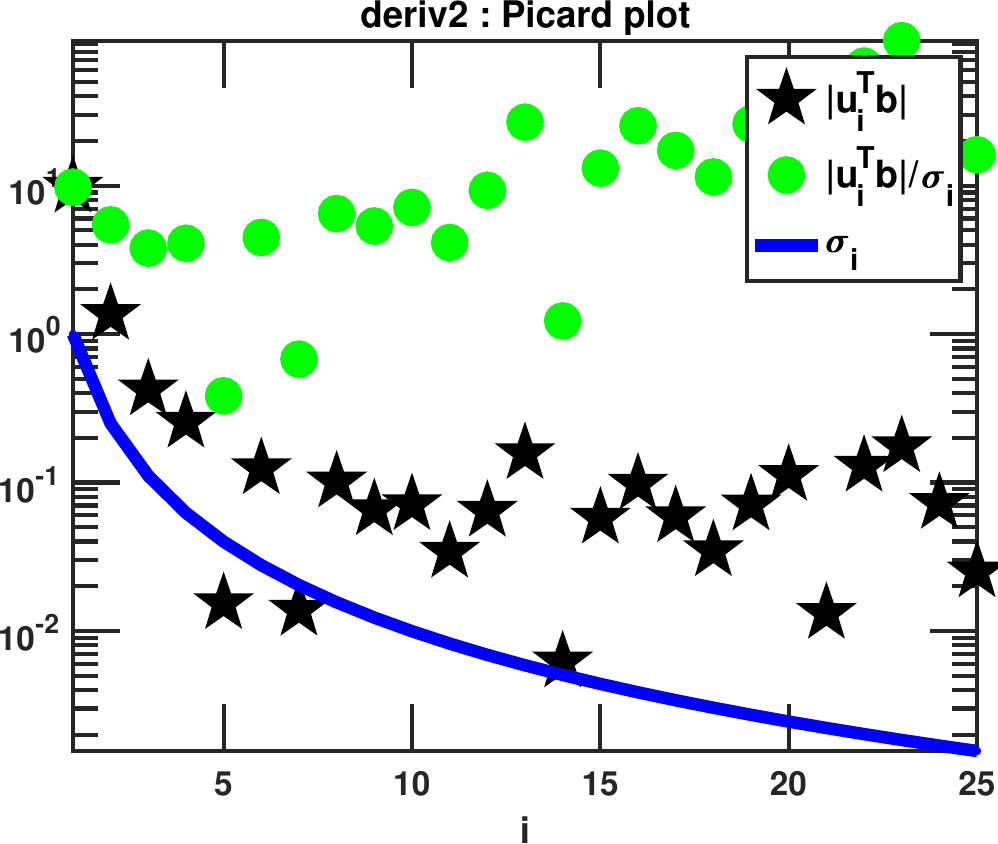}}
\end{center}
\caption{ Example \texttt{deriv2} from \cite{Regtools} showing the convergence of $\{\alpha_k\}$ for UPRE and GCV functions for TSVD sizes of $1:25$ as compared to the decay of the singular values, for the original problem of size $128$ and the associated Picard plot for the data. In Figures~\ref{fig:alphalowderiv2}-\ref{fig:lowpicardderiv2},  and  Figures~\ref{fig:alphahighderiv2}-\ref{fig:highpicardderiv2}, the noise variances are  $\sigma^2\approx 1\mathrm{e}{-4}$ and $\sigma^2\approx 1\mathrm{e}{-2}$, respectively. The converged relative errors in each case are $.273$ and $.381$ for the two noise variances $\sigma^2\approx 1\mathrm{e}{-4}$ and $\sigma^2\approx 1\mathrm{e}{-2}$, respectively.
\label{fig:functionsderiv2}}
\end{figure}

\begin{figure}[!htb]
\begin{center}
\subfloat[$\alpha_k$ convergence.  $\sigma^2\approx 1\mathrm{e}{-4}$. \label{fig:alphalowgravity}]{\includegraphics[width=.25\textwidth]{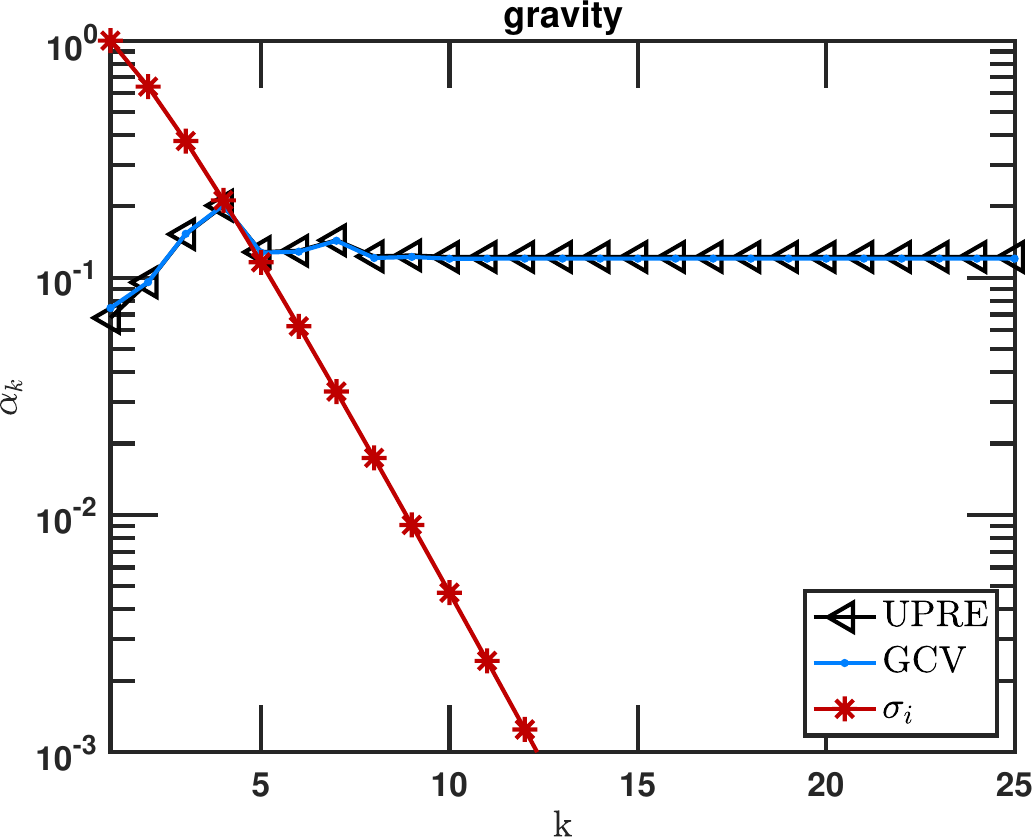}}
\subfloat[Picard Plot. $\sigma^2\approx 1\mathrm{e}{-4}$. \label{fig:lowpicardgravity}]{\includegraphics[width=.24\textwidth]{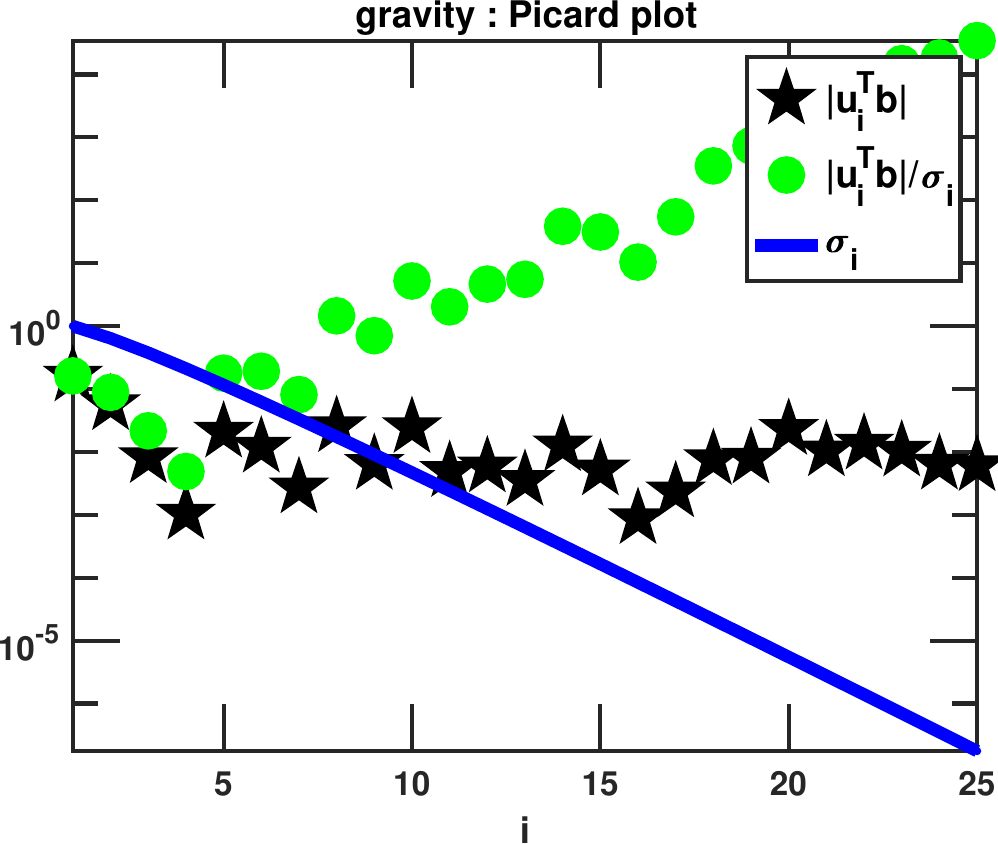}}
\subfloat[$\alpha_k$ convergence. $\sigma^2\approx 1\mathrm{e}{-2}$.\label{fig:alphahighgravity}]{\includegraphics[width=.25\textwidth]{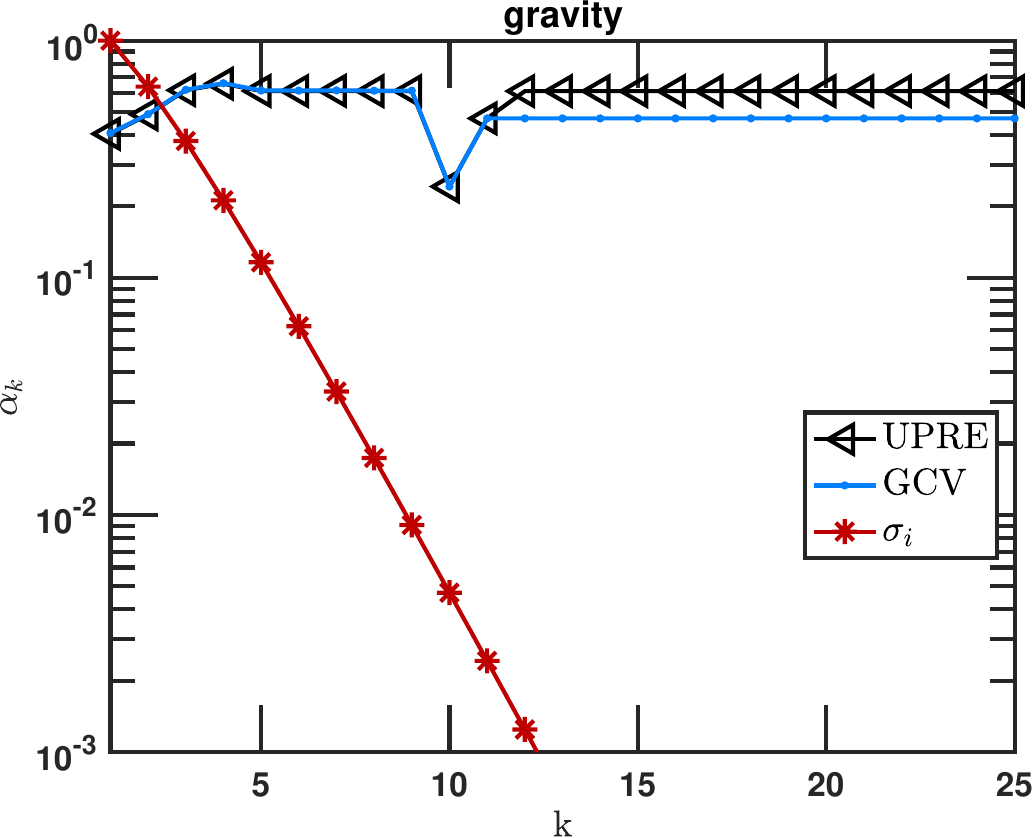}}
\subfloat[Picard Plot. $\sigma^2\approx 1\mathrm{e}{-2}$.\label{fig:highpicardgravity}]{\includegraphics[width=.24\textwidth]{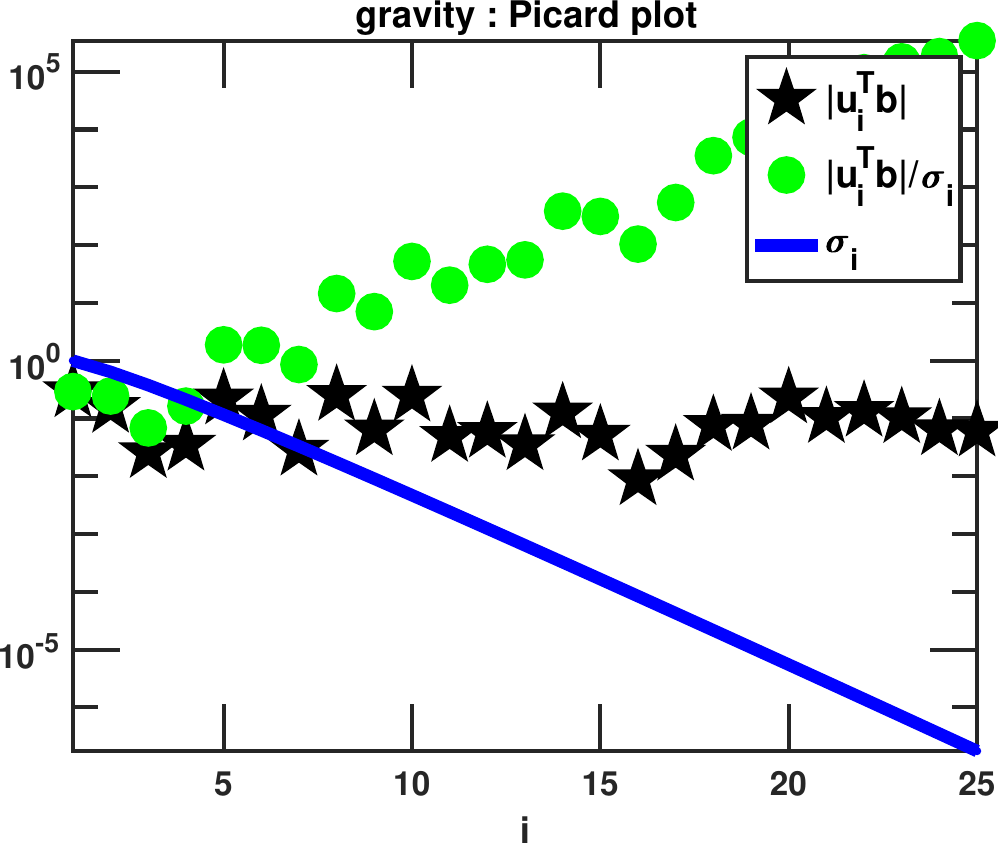}}
\end{center}
\caption{ Example \texttt{gravity} from \cite{Regtools} showing the convergence of $\{\alpha_k\}$ for UPRE and GCV functions for TSVD sizes of $1:25$  as compared to the decay of the singular values, for the original problem of size $128$ and the associated Picard plot for the data. In Figures~\ref{fig:alphalowgravity}-\ref{fig:lowpicardgravity},  and Figures~\ref{fig:alphahighgravity}-\ref{fig:highpicardgravity},  the noise variances are  $\sigma^2\approx 1\mathrm{e}{-4}$ and $\sigma^2\approx 1\mathrm{e}{-2}$, respectively. The converged relative errors  for noise variance $\sigma^2\approx 1\mathrm{e}{-4}$ are $.644$ and $.655$ for UPRE and GCV respectively. For noise variance $\sigma^2\approx 1\mathrm{e}{-2}$, these errors are $.787$ and $1.143$, respectively.\label{fig:functionsgravity}}
\end{figure}

Obtaining one-dimensional results, as shown in Figures~\ref{fig:functionsderiv2}-\ref{fig:functionsgravity}, is trivial but motivates the theoretical study of convergence in Section~\ref{sec:theory}, and then the application of that theory to standard two-dimensional problems in Section~\ref{sec:simulations}.

\section{Theoretical Results\label{sec:theory}}
We aim to find effective practical bounds on the regularization parameter $\alpha$ when found using the UPRE function. Observe first that we would not expect the regularization parameter to be larger than $\sigma_1$, otherwise all filter factors are less than $1/2$. Indeed imposing $\alpha=\sigma_1$ would lead to over smoothed solutions, and all of the dominant singular value components (the components without noise contamination) would be represented in the solution with filtering e.g \cite[Sections 4.4,  4.7]{hansenbook}.  In particular, the norm of the covariance matrix for the truncated filtered Tikhonov solution, the \textit{a posteriori} covariance of the solution, is approximately bounded by $\sigma^2/(4\alpha^2)$ which suggests smooth solutions for large $\alpha$. In contrast, the approximate bound for the \textit{a posteriori}  covariance when using the TSVD with $k$ terms without filtering is given by $\sigma^2/\sigma_k^2$  \cite[Sections 4.4.2, 4.4]{hansenbook}. Thus the filtered TSVD solution will be smoother than the TSVD solution when $\alpha>\sigma_k$:  increasing $\alpha$ reduces the covariance but provides more smoothing. Practically it is reasonable to impose the upper  bound $\alphamax\le \sigma_1=1$ for $\alpha$. To limit the noise that can enter the solution it is also desirable to find the lower bound  $\alphamin$. Solutions obtained for $\alpha\in[\alphamin,\alphamax]$, dependent on the spectrum of $A$,   should be sufficiently filtered but retain relatively unfiltered dominant components of the solution. We proceed to determine $\alphamin$ and to give a convergence analysis for $\alpha_k$ as the number of terms in the TSVD is increased.
\subsection{Convergence of $\{\alpha_k\}$ calculated using UPRE}%\label{upreconv}}
Denote  the UPRE function \eqref{uprefunc} for the rank $r$ problem by $U(\alpha)=U_r(\alpha)$ and the  optimal $\alpha$ for the filtered TSVD solution with $k$ components on the given interval as
\begin{eqnarray}\label{argmin}
\alpha_k=\mathrm{argmin}_{\alpha \in [\alphamin, \alphamax] } \, U_k(\alpha).
\end{eqnarray}
Ideally it would be helpful to find an interval  $[\alphamin, \alphamax]$ in which $U_k(\alpha)$ is strongly convex, but we have not been able to show this in general. Instead,  in the following we show that a  useful estimate of $\alphamin$ can be found.

For ease of notation within proofs  we use $\phi_i$ and $\gamma_i$ to indicate $\phi_i(\alpha)$ and $\gamma_i(\alpha)$, respectively, and denote differentiation  of a function $f(\alpha)$ with respect to $\alpha$ as $f^\prime$.

\begin{prop}%\label{prop:dubounds}
The following equalities are required for the future discussion. 
\begin{eqnarray}\label{dukalpha}
\frac{\partial U_k}{\partial \alpha} &=& \frac{4}{\alpha}\left(\sum_{i=1}^k s_i^2 \phi^2_i(\alpha) \gamma_i(\alpha) -{\sigma^2} \sum_{i=1}^k\phi_i(\alpha)\gamma_i(\alpha) \right),
\end{eqnarray}
\begin{eqnarray}\label{d2ukalpha}
\frac{\partial^2U_k}{\partial \alpha^2}&=&-\frac{1}{\alpha}\frac{\partial U_k}{\partial \alpha} +  \frac{8}{\alpha^2} \left( \sum_{i=1}^k s_i^2 \phi^2_i(\alpha) \gamma_i(\alpha) (2 \gamma_i(\alpha) - \phi_i(\alpha))  - \nonumber\right.\\
& &\left.\sigma^2 \sum_{i=1}^k \phi_i(\alpha) \gamma_i(\alpha)(\gamma_i(\alpha)-\phi_i(\alpha))\right).
\end{eqnarray}
\end{prop}
\begin{proof}
We use
\begin{equation*}%\label{deriv}
 \gamma_i^\prime = -\frac{2 \alpha\sigma_i^2}{(\sigma_i^2+\alpha^2)^2}=-\frac{2}{\alpha}\phi_i\gamma_i=-
 \phi_i^\prime <0.
\end{equation*}
Directly differentiating $U_k(\alpha)$  gives \eqref{dukalpha}
\begin{eqnarray*}
  U^{\prime}_k  &=&\sum_{i=1}^k s_i^2 2 \phi_i \phi_i^{\prime} +2 \sigma^2 \sum_{i=1}^k \gamma_i^{\prime} =\frac{4}{\alpha} \left(\sum_{i=1}^k s_i^2 \phi^2_i \gamma_i - \sigma^2 \sum_{i=1}^k \gamma_i \phi_i\right).
\end{eqnarray*}
Likewise for  the second derivative 
\begin{eqnarray}\label{firststep}
 U^{\prime\prime}_k &=&-\frac{1}{\alpha}  U^{\prime}_k + \frac{4}{\alpha} \left( \sum_{i=1}^k s_i^2 (2 \phi_i \phi_i^{\prime}\gamma_i + \phi^2_i\gamma_i^{\prime})-\sigma^2 \sum_{i=1}^k (\phi_i^{\prime}\gamma_i + \phi_i \gamma_i^{\prime})\right),
\end{eqnarray}
giving \eqref{d2ukalpha} after substitution for the derivatives.
\end{proof}
\begin{prop}\label{nomax}
Suppose that $0<\bar{\alpha}< {\sigma_k}/{\sqrt{2}}$ is a stationary point for $U_k(\alpha)$, for any $1\le k\le r$. Then $\bar{\alpha}$ is a unique minimum for $U_k({\alpha})$ on the interval  $0<\bar{\alpha}< {\sigma_k}/{\sqrt{2}}$.

\end{prop}
\begin{proof}
Removing the first term from \eqref{firststep}, identically zero at $\alpha=\bar{\alpha}$ by assumption that $\bar{\alpha}$ is  a stationary point,  gives
\begin{eqnarray*} \nonumber
\frac{\partial^2U_k}{\partial \alpha^2}(\bar{\alpha})&=& \frac{8}{\bar{\alpha}^2} \left( \sum_{i=1}^k s_i^2 \phi^2_i(\bar{\alpha}) \gamma_i(\bar{\alpha}) (2 \gamma_i(\bar{\alpha}) - \phi_i(\bar{\alpha}))  - \sigma^2 \sum_{i=1}^k \phi_i(\bar{\alpha}) \gamma_i(\bar{\alpha})(\gamma_i(\bar{\alpha})-\phi_i(\bar{\alpha}))\right) \\
&=& \frac{8}{\bar{\alpha}^2}\left( \sum_{i=  1}^k s_i^2  \phi^2_i(\bar{\alpha})\gamma_i(\bar{\alpha})(2 -3\phi_i(\bar{\alpha}))-\sigma^2 \sum_{i=  1}^k \phi_i(\bar{\alpha})\gamma_i(\bar{\alpha})(1-2\phi_i(\bar{\alpha}))\right).
\end{eqnarray*}
Now we substitute for $\sum_{i=1}^k s_i^2 \phi^2_i(\bar{\alpha}) \gamma_i(\bar{\alpha}) = \sigma^2 \sum_{i=1}^k \gamma_i(\bar{\alpha}) \phi_i(\bar{\alpha})$ using   \eqref{dukalpha}  at $\bar{\alpha}$  and note all terms are positive for $1-3\phi_i(\bar{\alpha})>0$, $i=1:k$. But $\phi_i$ is increasing with $i$ due to the ordering of the $\sigma_i$. Thus $1-3\phi_i(\bar{\alpha})\ge  1-3\phi_k(\bar{\alpha})>0$ for $\bar{\alpha}<\sigma_k/\sqrt{2}$ and $U_k^{\prime\prime}(\bar{\alpha})>0$. This result is true for any stationary point $\bar{\alpha}$ on the interval. Hence $U_k(\bar{\alpha})$  is a minimum for $U_k(\alpha)$ and it is only possible to have a maximum at $\alpha=0$, the end point of the given interval, but the end point is explicitly excluded from consideration. There are therefore no other stationary points within the interval and the minimum is unique. 
\end{proof}
\begin{remark} Although a minimum must exist in $[0, \sigma_k/\sqrt{2}]$ because $U_k(\alpha)$ is a continuous function on a compact set,   this result does not show that a minimum exists in $(0, \sigma_k/\sqrt{2})$ . \end{remark}

The next steps in the analysis rely on the following Assumptions \ref{assume1}-\ref{assume2} about the model and the data. 
\begin{assumption*}[Decay Rate\label{assume1}, \cite{Hansen:98,hofmann1986regularization}] The measured coefficients decay according to
$s^2_i=\sigma_i^{2(1+\nu)}> \sigma^2$  for $0<\nu<1$, $1\le i \le \ell$, i.e. the dominant measured coefficients follow the decay rate of the exact coefficients.
\end{assumption*}
\begin{assumption*}[Noise in Coefficients\label{assume2}]
There exists $\ell$ such that $E(s_i^2)=\sigma^2$ for all $i>\ell$, i.e. that the coefficients $s_i$ are noise dominated for $i>\ell$. Moreover, when $i\le \ell$ we assume that $E(s_i^2)\approx s_i^2$, so that the larger coefficients are effectively deterministic.
\end{assumption*}
These assumptions  have also been used in \cite{Hansen:98} for understanding how decay rates impact the convergence of iterative methods. 
We also recall that we use the non-restrictive  normalization $\sigma_1=1$ and use the notation $E(a)$ for the expectation of scalar deterministic $a$.

For the remaining results we  distinguish between the terms in the UPRE function  that are, and are not, contaminated by noise.
\begin{prop}\label{propukprime}
Suppose  Assumption~\ref{assume2} holds, then for $r> k+1>\ell$ there is a an upper bound on $E(U_k^\prime)$ independent of $\alpha$: 
\begin{eqnarray}\label{expectdu}
E(\frac{\partial U_{r}}{\partial \alpha} ) <  \dots <E(\frac{\partial U_{k+1}}{\partial \alpha} ) &<& E(\frac{\partial U_k}{\partial \alpha} ) < \frac{\partial U_{\ell}}{\partial \alpha} \,\,  \forall \,\, \alpha.
\end{eqnarray}
The lower bound for $E(U_k^{\prime\prime})$ holds for a fixed lower bound on $\alpha$
\begin{eqnarray}  \label{expectd2up}
E(\frac{\partial^2 U_{r}}{\partial \alpha^2} )> \dots > E(\frac{\partial^2 U_{k+1}}{\partial \alpha^2} )&>&E(\frac{\partial^2 U_k}{\partial \alpha^2} )> \frac{\partial^2 U_{\ell}}{\partial \alpha^2} \text{ if } \alpha>\frac{\sigma_{\ell+1}}{\sqrt{5}}, 
\end{eqnarray} 
whereas the upper bound depends also on an upper bound on $\alpha$ that decreases with increasing $k$
\begin{eqnarray} \label{expectd2un}
 E(\frac{\partial^2 U_{k+1}}{\partial \alpha^2})&<&  E(\frac{\partial^2 U_k}{\partial \alpha^2})    \,\, \text{if }   \,\, \alpha<\frac{\sigma_{k+1}}{\sqrt{5}}.
 \end{eqnarray}

\end{prop}
\begin{proof}
We note that the expectation operator is linear and when $a$ is not a random variable $E(a)=a$. Applying these properties first to   \eqref{dukalpha} yields
\begin{eqnarray*}%\label{Edukalpha}
E(  U^\prime_k )
&=& E\left( U^\prime_{\ell}  +\frac{4}{\alpha}\sum_{i=\ell+1}^k\phi_i\gamma_i (s_i^2 \phi_i -\sigma^2)\right)\\ %\nonumber
&\approx &  U^\prime_{\ell} +  \frac{4\sigma^2}{\alpha}\sum_{i=\ell+1}^k\phi_i\gamma_i (\phi_i -1)  <  U^\prime_{\ell},%\nonumber
\end{eqnarray*}
where from line one to two we use linearity, and, by Assumption~\ref{assume2}, $E(U^\prime_{\ell})= U^\prime_{\ell}$ and $E(s_i^2)=\sigma^2$ for $i>\ell$. 
In particular, in expectation each term for $i>\ell$ is negative and recursively both inequalities in \eqref{expectdu} apply.
Applying the expectation operator now to \eqref{d2ukalpha}  gives
\begin{eqnarray*}
E(  U^{\prime\prime}_k  )&\approx & \frac{\partial^2 U^{\prime\prime}_{\ell}}{\partial \alpha^2} +\frac{4\sigma^2}{\alpha^2}\left(\sum_{i=\ell+1}^k\phi_i\gamma_i (1-\phi_i ) +2 \left(\phi_i^2 \gamma_i (2 \gamma_i-\phi_i)-\phi_i\gamma_i(\gamma_i-\phi_i)\right)\right)\\
&=&   U^{\prime\prime}_{\ell}  +\frac{4\sigma^2}{\alpha^2}\left(\sum_{i=\ell+1}^k\phi_i\gamma_i \left(1-\phi_i  +2 \left(\phi_i  (2 \gamma_i-\phi_i)-(\gamma_i-\phi_i)\right)\right)\right)\\
&=&    U^{\prime\prime}_{\ell}  +\frac{4\sigma^2}{\alpha^2}\left(\sum_{i=\ell+1}^k\phi_i\gamma_i \left(1 -\phi_i  +2 \left(\phi_i  (2 -3 \phi_i)-(1-2\phi_i)\right)\right)\right)\\
&=&    U^{\prime\prime}_{\ell}  +\frac{4\sigma^2}{\alpha^2}\left(\sum_{i=\ell+1}^k\phi_i\gamma_i \left(-6\phi^2_i +7\phi_i  -1\right)\right).
\end{eqnarray*}
The sign of the second term depends on the sign of $-6\phi^2_i +7\phi_i  -1 $ which is increasing from $-1$ as a function of $\phi\le 1$. Hence
\begin{eqnarray*}
-6\phi^2_i +7\phi_i  -1 \left\{ \begin{array}{llcl}
\ge   -6\phi^2_{\ell+1} +7\phi_{\ell+1}  -1 & =\frac{\sigma_{\ell+1}^2(5 \alpha^2-\sigma_{\ell+1}^2)}{(\alpha^2+\sigma_{\ell+1}^2)^2}&>  0 & \text{ if } \alpha >\frac{\sigma_{\ell+1}}{\sqrt{5}} \\
\le  -6\phi^2_{k} +7\phi_{k}  -1 & =\frac{\sigma_{k}^2(5 \alpha^2-\sigma_{k}^2)}{(\alpha^2+\sigma_{k}^2)^2}& < 0 &\text{ if } \alpha < \frac{\sigma_{k}}{\sqrt{5}}. \end{array}\right.
\end{eqnarray*}
Again, in expectation, terms for $i>\ell$ are all positive when $\alpha\ge {\sigma_{\ell+1}}/{\sqrt{5}}$ and the nested inequalities in \eqref{expectd2up} apply. The requirement that the $i^{th}$ term is necessarily positive becomes more severe as $i$ increases, yielding the additional  inequality with conditions on $\alpha$ given in  \eqref{expectd2un}.
\end{proof}

\begin{corollary}\label{cor2}
Suppose  Assumption~\ref{assume2} holds, and that for $\alpha_{\ell}>{\sigma_{\ell+1}}/{\sqrt{5}}$,     $U_{\ell}(\alpha_\ell)$ is a minimum for $U_\ell(\alpha)$. Then for $\ell<k\le r$, $U_k(\alpha)$ is convex and decreasing at $\alpha_\ell$,
$$
 E(\frac{\partial U_k(\alpha_{\ell})}{\partial \alpha} )<0  \quad \text{and} \quad E(\frac{\partial^2 U_k(\alpha_{\ell})}{\partial \alpha^2} )>  0.
$$
\end{corollary}
\begin{proof}
If $U_\ell(\alpha_{\ell})$ is  a minimum, then  $ U^\prime_\ell(\alpha_{\ell})  =0$ and $   U^{\prime\prime}_{\ell}  (\alpha_{\ell})>0$ and the inequalities follow immediately from \eqref{expectdu} and \eqref{expectd2up}.
\end{proof}
\begin{corollary}\label{nomin}
Suppose  Assumption~\ref{assume2} holds. If a stationary point $\alpha_r < \sigma_r/\sqrt{5}$ exists there are no stationary points of $U_k(\alpha)$ for $\alpha \in (\sigma_r/\sqrt{5},\sigma_k/\sqrt{2})$. 
\end{corollary}
\begin{proof}
Suppose that $\alpha_r \in [0,  \sigma_r/\sqrt{5})$. The existence of $\alpha_r$ in this interval  does not contradict  Proposition~\ref{nomax} since  $\sigma_r/\sqrt{5}<\sigma_r/\sqrt{2}$.  By assumption,  $U^{\prime}_r(\alpha_r)=0$ and $U^{\prime\prime}_r(\alpha_r)>0$. Thus  by    \eqref{expectdu} and \eqref{expectd2un} $U^{\prime}_k(\alpha_r)>0$ and $U^{\prime\prime}_k(\alpha_r)>0$, and  $U_k(\alpha)$, $\ell \le k \le r-1$ is convex and increasing at $\alpha_r$.  Therefore, by continuity, $U_k(\alpha)$ cannot reach a minimum for $\alpha_r<\alpha_k< \sigma_k/\sqrt{2} $ without first passing through a stationary point which is a maximum. But by Proposition~\ref{nomax} there is no maximum of $U_k(\alpha)$ to the left of $\sigma_k/\sqrt{2}$ and thus there is also no minimum for $\alpha_r<\alpha< \sigma_k/\sqrt{2} $.  In particular $U_k(\alpha)$ has no stationary point for $\sigma_r/\sqrt{5}\le \alpha\le \sigma_k/\sqrt{2}$.
\end{proof}
\begin{remark}
We have shown through Corollary~\ref{nomin} that if  $U_r(\alpha_r)$ is a minimum for $U_r(\alpha)$ and $\alpha_r < \sigma_r/\sqrt{5}$ then $U_k(\alpha_k)$ can only be a minimum for $U_k(\alpha)$ if  either $\alpha_k\le \alpha_r\le \sigma_r/\sqrt{5}$ or $\alpha_k>\sigma_k/\sqrt{2}$, i.e. we may require  $\alpha_k > \sigma_k/\sqrt{2}$ under the assumption that we seek $\alpha_r>\sigma_r$. This applies for all $k$ with $1\le \ell \le k\le r-1$.
\end{remark}
Although this result does provide a refined lower bound for $\alpha_k$, it is dependent on $k$ and decreasing with $k$, which is not helpful when $k$ gets large, as needed for finding $ \alpha_r$, i.e. this bound would suggest that $ \alpha_r$ needs to be found  using the pessimistic lower bound $\sigma_r/\sqrt{2}$. We investigate now whether these lower bounds on $\alpha$ are indeed realistic by looking for  bounds on the UPRE functions $U_k(\alpha)$.
\begin{prop}\label{boundonsU}
Suppose Assumptions~\ref{assume1} and \ref{assume2} hold,  then lower and upper bounds on $U_k(\alpha)$ and its derivatives are given by $\mathcal{L}_k(\alpha)$ and $\mathcal{U}_k(\alpha)$ and their derivatives, respectively, where
 \begin{align}\label{Ubounds}
 0<\mathcal{L}_k(\alpha)=G(\alpha)+ F_k(\alpha)&< E(U_k(\alpha))<H(\alpha)+F_k(\alpha)=\mathcal{U}_k(\alpha)  \\  %\label{G1G2G3}
\mathcal{L}_k^\prime(\alpha)= G^\prime(\alpha)+ F^{\prime}_k(\alpha)&< E(U^{\prime}_k(\alpha))<H^{\prime}(\alpha)+F^{\prime}_k(\alpha)=\mathcal{U}_k^\prime(\alpha) , \label{Uprimebounds} \\
\mathcal{L}_k^{\prime\prime}(\alpha)= G^{\prime\prime}(\alpha)+ F^{\prime\prime}_k(\alpha)&< E(U^{\prime\prime}_k(\alpha))<H^{\prime\prime}(\alpha)+F^{\prime\prime}_k(\alpha)=\mathcal{U}_k^{\prime\prime}(\alpha), \text{ for } \alpha\le\sigma_\ell \text{ but } \label{U2primebounds} \\\nonumber%\label{U2primebounds2}
\mathcal{U}_k^{\prime\prime}(\alpha)= H^{\prime\prime}(\alpha)+ F^{\prime\prime}_k(\alpha)&< E(U^{\prime\prime}_k(\alpha))<G^{\prime\prime}(\alpha)+F^{\prime\prime}_k(\alpha)=\mathcal{L}_k^{\prime\prime}(\alpha), \text{ for } \alpha>1.
\end{align}
Here $G(\alpha)$ and $H(\alpha)$ are independent of $k$, while $F_k(\alpha)$ very clearly depends on the $k$ terms in the sums as given by
\begin{eqnarray}
G(\alpha)&=&  \alpha^4 \sum_{i=1}^{\ell} \gamma^2_i,\, \, H(\alpha)= \alpha^2 \sum_{i=1}^{\ell} \phi_i \gamma_i, \,\label{H12} \text{ and }\\
 F_k(\alpha)&=&\sigma^2 \left((k-\ell) + 2\sum_{i=1}^{\ell}\gamma_i + \sum_{i=\ell+1}^{k} \gamma_i^2  \right).\label{H3}
\end{eqnarray}
 \end{prop}
\begin{proof}
By  \eqref{Picard} due to Assumption~\ref{assume1} for $i\le\ell$
\begin{eqnarray}\label{boundsi1}
\sigma_i^4<\sigma_i^{2(1+\nu)}=s_i^2 <\sigma_i^2.
\end{eqnarray}
Thus
\begin{eqnarray}\label{boundsi2}
\alpha^4 \gamma_i^2=\sigma_i^4 \phi_i^2 <\phi_i^2(\alpha) s_i^2< \sigma_i^2 \phi_i^2 = \alpha^2 \phi_i \gamma_i.
 \end{eqnarray}
Now from \eqref{uprefunc}
\begin{eqnarray*}
E(U_{\ck}(\galpha))&=& \sum_{i=1}^{\ell} \phi_i^2 s_i^2+\sigma^2  (2\sum_{i=1}^{\ck}\gamma_i(\galpha)+\sum_{i=\ell+1}^{k} \phi_i^2) =\sum_{i=1}^{\ell} \phi_i^2 s_i^2+F_k(\alpha),
\end{eqnarray*}
may be bounded  using \eqref{boundsi1}. This   yields immediately \eqref{Ubounds} with the noted definitions for  $G$, $H$ and $F_k$, as given in \eqref{H12}-\eqref{H3}.

To show \eqref{Uprimebounds} introduce $D_i(\alpha)>0$, $i=1$, $2$, given by
 \begin{eqnarray*}
 D_1(\alpha) &=& E(U_k(\alpha)) -(G(\alpha)+ F_k(\alpha)) =\sum_{i=1}^\ell( \phi_i^2s_i^2 -\alpha^4 \gamma_i^2) \\
  D_2(\alpha) &=& (H(\alpha)+ F_k(\alpha))-E(U_k(\alpha)) =\sum_{i=1}^\ell( \alpha^2 \phi_i\gamma_i -\phi_i^2s_i^2 ).
 \end{eqnarray*}
Then $D_i$ are independent of $k$ and
 \begin{eqnarray*} 
 D^{\prime}_1(\alpha)&=&\sum_{i=1}^\ell \left(\frac{2}{\alpha}( 2\phi^2_i\gamma_is_i^2 +2 \alpha^4    \gamma^2_i \phi_i)  -4\alpha^3\gamma_i^2\right)  =\frac{4}{\alpha}  \sum_{i=1}^\ell (\phi^2_i\gamma_is_i^2 - \alpha^4 \gamma^3_i )\text{ and } \\  
  D^{\prime}_2(\alpha)&=&\sum_{i=1}^\ell( 2\alpha  \phi_i\gamma_i +  \frac{2}{\alpha} (\alpha^2\phi_i \gamma_i (1-2\phi_i) -2 \phi_i^2\gamma_is_i^2) ) =\frac{4}{\alpha} \sum_{i=1}^\ell (\alpha^2 \phi_i\gamma_i^2 -s_i^2 \phi_i^2\gamma_i).
 \end{eqnarray*}
But now again applying  Assumption~\ref{assume1} we have
\begin{eqnarray}\label{boundsi3}
 \alpha^4  \gamma_i^3=\sigma_i^4 \phi^2_i\gamma_i < {s^2_i} \phi^2_i\gamma_i < \sigma_i^2 \phi^2_i\gamma_i = \alpha^2 \phi_i \gamma_i^2.\end{eqnarray}
 Therefore $D_i^{\prime}(\alpha)>0$, $i=1$, $2$ and we immediately obtain \eqref{Uprimebounds}.

 The second derivative result follows similarly using
 \begin{eqnarray*} 
 D^{\prime\prime}_1(\alpha)
 &=& \frac{12}{\alpha^2}\sum_{i=1}^\ell   \gamma_i (1-2\phi_i)(s_i^2\phi_i^2 -\alpha^4\gamma_i^2)>0 \\
  D^{\prime\prime}_2(\alpha)
   &=& \frac{12}{\alpha^2}\sum_{i=1}^\ell   \phi_i\gamma_i (1-2\phi_i)(\gamma_i\alpha^2 - s_i^2\phi_i)>0,
 \end{eqnarray*}
 where in each case we apply \eqref{boundsi3} and note $1-2\phi_i\ge 0$, for $1\le i\le \ell$ and $\alpha\le \sigma_\ell$. This then immediately gives the reverse inequalities for $\alpha>1$.
  \end{proof}

From \eqref{H12}-\eqref{H3} we see that we may write $G$, $H$ and $F_k$  in terms of sums $S_p(i_1,i_2)=\sum_{i=i_1}^{i_2} \gamma^p_i$ for $p=1$ and $p=2$ by writing $\phi_i\gamma_i=\gamma_i-\gamma_i^2$. Hence 
\begin{align*}%\label{Hsums}
G(\alpha)=\alpha^4 S_2(1,\ell), \,\, H(\alpha)=\alpha^2(S_1(1,\ell)-S_2(1,\ell)) \quad \mathrm{and}\\
 \,\, F_k(\alpha)=\sigma^2 ( k-\ell +2S_1(1,\ell)+S_2(\ell+1,k)  ). %\label{Hsum3}
\end{align*}
 Thus for $U_\ell(\alpha)$ we have the bounding functions by Proposition~\ref{propukprime}
\begin{eqnarray*}
\mathcal{L}_\ell(\alpha)&=&G(\alpha)+F_\ell(\alpha)=\alpha^4 S_2+ 2\sigma^2S_1 \\
\mathcal{U}_\ell(\alpha)&=&H(\alpha)+F_\ell(\alpha)=\alpha^2(S_1-S_2)+ 2\sigma^2S_1,
\end{eqnarray*}
where the sums all range from $1$ to $\ell$.
Moreover, also by Proposition~\ref{propukprime}, $\mathcal{L}_\ell^\prime(\alpha)<U_\ell(\alpha)<\mathcal{U}_\ell^\prime(\alpha)$ where
\begin{eqnarray*}
\mathcal{L}_\ell^\prime(\alpha)&=&4\alpha^3 S_2+ \alpha^4S_2^\prime + 2\sigma^2S^\prime_1 = 4\alpha^3 (S_2 +S_3-S_2)+\frac{4\sigma^2}{\alpha}(S_2-S_1)\\
&=&\frac{4}{\alpha}(\alpha^4 S_3+\sigma^2(S_2-S_1)) \\
\mathcal{U}_\ell^\prime(\alpha)&=&2\alpha(S_1-S_2)+ \alpha^2(S_1^\prime-S_2^\prime) +2\sigma^2S^\prime_1 \\
&=& 2\alpha(S_1-S_2)+2\alpha(S_2-S_1 -2(S_3-S_2))+\frac{4}{\alpha}\sigma^2(S_2-S_1)\\
&=&\frac{4}{\alpha}(\alpha^2(S_2-S_3)+\sigma^2(S_2-S_1)),
\end{eqnarray*}
and we used $\gamma_i^\prime = -({2}/{\alpha})\gamma_i\phi_i=({2}/{\alpha})(\gamma_i^2-\gamma_i)$ and $(\gamma_i^2)^\prime=-({4}/{\alpha}) \gamma^2_i\phi_i=({4}/{\alpha})(\gamma_i^3-\gamma_i^2)$.
\begin{prop}\label{limitUlprime} Suppose Assumption~\ref{assume1} holds, then necessarily
$U_\ell^\prime(\alpha)<0$ for $\alpha^2<\sigma_{\ell+1}^2/(1-\sigma_{\ell+1}^2)$. Hence   $\alpha^2_{\ell}>\sigma^2_{\ell+1}/(1-\sigma^2_{\ell+1})$.
\end{prop}
\begin{proof}
If the upper bound  has a negative slope,  $ \mathcal{U}_\ell^\prime(\alpha)<0$ for some $\alpha$,  then  $U_\ell^\prime(\alpha) <0$ also. Immediately $ \mathcal{U}_\ell^\prime(\alpha)<0$ for $\alpha^2(S_2-S_3)+\sigma^2(S_2-S_1)<0$, and for $U_\ell^\prime(\alpha)<0$ it is sufficient  that
for $1 \le i\le\ell$
\begin{eqnarray*}
0&>& \alpha^2(\gamma_i^2-\gamma_i^3)+\sigma^2(\gamma_i^2-\gamma_i)=\gamma_i( \alpha^2\gamma_i(1-\gamma_i)+\sigma^2(\gamma_i-1))=\gamma_i\phi_i( \alpha^2\gamma_i-\sigma^2),
\end{eqnarray*}
and we need $( \alpha^2\gamma_i-\sigma^2)<0$, or $\alpha^2\sigma_i^2-\sigma^2(\alpha^2+\sigma_i^2)<0$. Now, for $i\le \ell$, $\sigma_i^2\ge \sigma_\ell^2>\sigma^2$ and we obtain
$\alpha^2< \min({\sigma^2\sigma_i^2}/{(\sigma_i^2-\sigma^2)})$ for all $1\le i\le\ell$.  But $x^2/(x^2-a^2)$ is decreasing with $x$ for $x^2>a^2$, hence we need $\alpha^2<\sigma^2/(1-\sigma^2)$. For $\sigma_{\ell+1}^2<\sigma^2<\sigma_\ell^2$  and using $x^2/(1-x^2)$, which is increasing with $x \in (0,1)$, we obtain $\alpha^2<\sigma_{\ell+1}^2/(1-\sigma_{\ell+1}^2)$. Hence we must have  $\alpha^2_{\ell}>\sigma^2_{\ell+1}/(1-\sigma^2_{\ell+1})$.
\end{proof}

We now extend the analysis to obtain a lower bound on $\alpha_k$ for all $k> \ell$.
\begin{theorem}\label{thm:minbnd}
Suppose Assumptions~\ref{assume1} and \ref{assume2} hold, and that $U_k(\alpha_k)$ is a minimum for $U_k(\alpha)$,  then, for $k \ge \ell$, $\alpha_k> \alpha_\ell>\sigma_{\ell+1}/\sqrt{1-\sigma_{\ell+1}^2}=\alphamin$.
\end{theorem}
\begin{proof}
First suppose the contrary and that $\alpha_k\le  \sigma_{\ell+1}/\sqrt{1-\sigma_{\ell+1}^2}$.  Then $U_k^\prime(\alpha_k)=0$ and by \eqref{expectdu} $U^\prime_\ell(\alpha_k)>0$.  But by Proposition~\ref{limitUlprime}  $U^\prime_\ell(\alpha)<0$ for $\alpha\le \sigma_{\ell+1}/\sqrt{1-\sigma_{\ell+1}^2}$ and we have a contradiction yielding
  $\alpha_k> \sigma_{\ell+1}/\sqrt{1-\sigma_{\ell+1}^2}=\alphamin$, $k\ge \ell$.  It remains  to determine whether it is possible to have $\sigma_{\ell+1}/\sqrt{1-\sigma_{\ell+1}^2}<\alpha_k<\alpha_\ell$ where $\alpha_\ell$ is the first minimum point of $U_\ell(\alpha)$ to the right of $\alphamin$. Again we proceed by contradiction and suppose that $\alpha_k \in [ \alphamin, \alpha_\ell]$ exists. Then we have the following:   \begin{enumerate}
\item
By \eqref{expectdu} $E(U_k^\prime)(\alpha_\ell)<U_\ell^\prime(\alpha_\ell)=0$,  and by \eqref{expectd2up}, noting  $\alpha>\sigma_{\ell+1}/\sqrt{5}$, $E(U^{\prime\prime}_k)(\alpha_\ell)> U^{\prime\prime}_\ell(\alpha_\ell)>0$. Hence $U_k(\alpha)$ is  convex and decreasing at $\alpha_\ell$.
\item At the minimum critical point $\alpha_k<\alpha_\ell$, $U_k^\prime(\alpha_k)=0$. Thus  there must also be a second critical point which is a maximum for some $\bar{\alpha}$ in the interval $\alpha_k<\bar{\alpha}<\alpha_\ell$, for which $U_k^\prime(\bar{\alpha})=0$ and $U_k^{\prime\prime}(\bar{\alpha})<0$.
\item At $\bar{\alpha}$ we then have by \eqref{expectdu} that
$U^\prime_\ell(\bar{\alpha})>0$. Hence  $U_\ell(\alpha)$ is increasing at $\bar{\alpha}<\alpha_\ell$  but is decreasing at $\alphamin<\bar{\alpha}$, i. e. $U_\ell^\prime(\alpha)$ changes sign for some $\alpha $ in the interval $[\alphamin, \bar{\alpha}]$. But by continuity then $U_\ell(\alpha)$ has at least one minimum on this interval.  By assumption, however, $\alpha_\ell$ is the first minimum point of $U_\ell(\alpha)$ to the right of $\alphamin$ and we have arrived at a contradiction.
\end{enumerate}
\end{proof}
We have now obtained a tight lower bound on $\alpha_k$
\begin{equation}\label{tightlowerbound}
\alphamin=\frac{\sigma_{\ell+1}}{\sqrt{1-\sigma^2_{\ell+1}}} < \alpha_k,  \quad \ell \le k \le r.
\end{equation}
It remains to discuss the convergence of $\{\alpha_k\}$ to $\alpha_{\kopt}$ with increasing $k$.
We note that one approach would be to show that the $U_k(\alpha)$ are convex for $\alpha>\sigma_\ell$, but the sign result in \eqref{U2primebounds} only immediately applies for $\alpha>1$, hence investigating the sign requires a more refined bound for each interval $\alpha \in [\sigma_{i}, \sigma_{i-1}]$ for $i\le \ell$. Instead we obtain the following result, which relies on the uniqueness of $\alpha_k$.

\begin{theorem}\label{thm:conv}
Suppose Assumptions~\ref{assume1} and \ref{assume2} hold and that $  \alpha_{\kopt}$ and  each $\alpha_k$, $k>\ell$ are unique  within the given interval $\sigma_{\ell+1}/\sqrt{1-\sigma^2_{\ell+1}}<\alpha<1$.  Then, the sequence $\{\alpha_k\}_{k>\ell}$ is on the average increasing with $\lim_{k\rightarrow r} E(\alpha_k)=E(   \alpha_{\kopt})$ and $\{U_k(\alpha_k)\}$ is increasing.
\end{theorem}
\begin{proof}
It is immediate from \eqref{uprefunc} that $U_k(\alpha)\ge U_{\ell}(\alpha)$ for any $k>\ell$ and any $\alpha$, and that $U_{k+1}(\alpha)\ge U_k(\alpha)$. Thus the $\{U_k(\alpha)\}$ is an increasing set of functions with $k>\ell$.
By \eqref{expectdu} of Proposition~\ref{propukprime} we also have $E(\frac{\partial U_{k+1}(\alpha)}{\partial \alpha} )< E(\frac{\partial U_k(\alpha)}{\partial \alpha} ) < \frac{\partial U_{\ell}(\alpha)}{\partial \alpha} $, and $\{E(\frac{\partial U_k(\alpha)}{\partial \alpha} ) \}$ is a decreasing set of functions for $k>\ell$. In particular   $E(\frac{\partial U_{k+1}(\alpha_{\ell})}{\partial \alpha} )< E(\frac{\partial U_k(\alpha_{\ell})}{\partial \alpha} )<0$. Moreover, by Corollary~\ref{cor2}  and \eqref{expectd2up} of Proposition~\ref{propukprime}, when $\alpha_{\ell}>{\sigma_{\ell+1}}/{\sqrt{5}}$ the expected second derivatives at $\alpha_{\ell}$ are positive and increasing with $k$ so that the first derivative increases to $0$ more quickly for larger $k$. Thus, not only do we have $E(\alpha_k)>\alpha_{\ell}>\alphamin$ for all $k$, we also have that  $\{E(\alpha_k)\}$ converges from below to  $E(  \alpha_{\kopt})$.
\end{proof}

\begin{corollary}[Faster Decay Rate of the Coefficients]
Suppose that the coefficients $s_i$ decay at the  rate $s_i^2 =\sigma_i^{2(\rho+\nu)}$ for integer $\rho>1$. Then the results of Theorems~\ref{thm:minbnd}-\ref{thm:conv} still hold.
\end{corollary}
\begin{proof}
This holds by modifying the inequality \eqref{Picard} for the faster decay rate yielding
\begin{eqnarray*}%\label{boundsi1faster}
K_i\sigma_i^4<\sigma_i^{2(\rho+\nu)}=s_i^2 <\sigma_i^2K_i, \quad K_i=\sigma_i^{2(\rho-1)}.
\end{eqnarray*}
Thus the coefficients are bounded as in \eqref{boundsi2} but with scale factor $K_i$
\begin{eqnarray*}%\label{boundsi2faster}
\alpha^4 \gamma_i^2K_i=\sigma_i^4 \phi_i^2K_i &<&\phi_i^2(\alpha) s_i^2<K_i \sigma_i^2 \phi_i^2 = K_i\alpha^2 \phi_i \gamma_i.
 \end{eqnarray*}
Using this relation all the results presented in Proposition~\ref{boundonsU} still hold with $H(\alpha)$ and $G(\alpha)$ replaced by
\begin{equation*}G_{\rho}(\alpha) = \alpha^4 \sum_{i=1}^\ell K_i \gamma^2_i, \quad \text{and}\quad H_{\rho}(\alpha) = \alpha^2 \sum_{i=1}^\ell K_i \phi_i\gamma_i.
\end{equation*}
Then again redefining the summations $S_p$ to now depend on the coefficients with $K_i$, for $H_\rho$ and $G_\rho$, following Proposition~\ref{limitUlprime}  yields the condition
$$\gamma_i\phi_i(\alpha^2 K_i \gamma_i -\sigma^2)<0$$ for $U^\prime_\ell(\alpha)<0$. Continuing the argument as in the proof of Proposition~\ref{limitUlprime} still yields the lower bound $\alpha_\ell^2 > \sigma_{\ell+1}^2/(1-\sigma_{\ell+1}^2)$. But this is all that is required for Theorems~\ref{thm:minbnd}-\ref{thm:conv} and hence the results follow without modification.
\end{proof}

\begin{remark}
This result shows that given a TSVD which sufficiently incorporates the dominant terms of the SVD expansion, including sufficient terms that are noise-contaminated, $\alpha_k$ will be an increasingly good approximation for $  \alpha_{\kopt}$. Moreover, including additional terms in the expansion will have limited impact on the solution, because $  \alpha_{\kopt}>\alpha_{\ell}$ and  filter factor $\gamma_i(  \alpha_{\kopt})$ is decreasing with $i$. In particular,  we are using  $\gamma_i(   \alpha_{\kopt})< \gamma_i(\alpha_{\ell})<\gamma_{\ell+1}(\alpha_{\ell})<\gamma_{\ell+1}(\sigma_{\ell+1})=1/2$, for $i>\ell+1$ and $\alpha_{\ell}>\sigma_{\ell+1}$. These nested inequalities follow immediately because $\gamma(x,y)=y^2(y^2+x^2)^{-1}$ is decreasing as a function of $x$ and increasing as a function of $y$. 
\end{remark}
\begin{remark}
Although the main result of this paper effectively relies on an assumption that the UPRE functions have unique minima within the obtained bounds, $\alphamin<\alpha_k<1$, proving that the minima are indeed unique seems to require using the discrete summations occurring in $U_k(\alpha)$  as approximations to continuous integrals. This approach is very technical, not very general, being dependent on the decay rate parameter $\tau$, and serves only to tighten the lower bound for $\alpha$. We therefore chose not to present results along this direction, relying on the computational results that are supportive of the unique identification of a minimum within these realistic bounds.
\end{remark}
\begin{remark}
The results given depend on the assumption that summations with $s_i^2$ for terms with $i>\ell$ may be approximated in terms of the noise variance. For $r-\ell$  small relative to $r$, this assumption breaks down. As $r-\ell$ increases the assumptions become more reliable and less impacted by outlier data for $s_i^2$. Still the main convergence theorem holds only with respect to this analysis and we cannot expect that $\{\alpha_k\}$ will always converge monotonically to $  \alpha_{\kopt}$ in practice. With sufficient safeguarding, as noted in the algorithm presented in Section~\ref{sec:simulations}, it is reasonable to expect that $  \alpha_{\kopt}$ is quickly and accurately identified.
\end{remark}

\begin{remark}[Posterior Covariance]
We have shown  $\{\alpha_k\}$ increases with $k$. Consequently, the approximate \textit{a posteriori} covariance of the filtered TSVD solution  $\sigma^2/(4\alpha^2)$ decreases with $k$, to $\sigma^2/(4   \alpha_{\kopt}^2)$. In trading-off the minimization of the risk by using the UPRE to find the optimal $\alpha$, the method naturally finds a solution which has increasing smoothness with increasing $k$. This limits the impact of the possibly non-smooth components of the solution corresponding to small singular values, most likely noise-contaminated,  that would contaminate the unfiltered TSVD solution.
\end{remark}

\section{Practical Application}\label{sec:simulations}

The convergence theory  for $\{\alpha_k\} \rightarrow  \alpha_{\kopt}$ as $k \rightarrow \kopt$ presented in Section~\ref{sec:theory} motivates the construction of an algorithm to automatically determine the optimal index $\kopt$, defined as in Section~\ref{sec:intro}  to be the optimal number of terms to use from the  TSVD,  and associated regularization parameter $ \alpha_{\kopt}$. The  algorithm is presented and discussed in Section~\ref{sec:alg} and tested for $2D$ test problems using IR Tools \cite{GaHaNa:IR} in Section~\ref{sec:algresults}. These results also corroborate the convergence theory presented in Section~\ref{sec:theory}.

\subsection{Algorithm}\label{sec:alg}
We propose an algorithm that works by iteratively minimizing \eqref{uprefunc} on the TSVD subspace of size $k \le r$ until a set of convergence criteria are met.  These convergence criteria rest on the observation that $\emph{in general}$ for sufficiently large $k$, the relative change, $c_k =|(\alpha_k-\alpha_{k+1})|/\alpha_k>0$,  between successive parameter estimates, $\alpha_{k}$ and $\alpha_{k+1}$, decreases as $k$ increases towards $r$. If during the iterative procedure there exists a $k$ such that it is reasonably believed that $\alpha_k \approx \alpha_i$ for all $i > k$, the algorithm terminates, producing $\kopt$ and $\alpha_{\kopt}$. A pseudo-code implementation is given as Algorithm~\ref{influx}.

\begin{algorithm}[!ht]
  \SetAlgoLined
  \DontPrintSemicolon
  \LinesNumbered
  \KwIn{\
      SVD or TSVD;
      data $\mb b$ and noise variance estimate $\sigma^2$;
      initial index $k_0$;
      maximum $k$, $k_\text{max}$;
      step size $\Delta_k$;
      relative tolerance $\delta$;
      window length $w$;
      optional estimate for $\ell$
      }
  \KwOut {\
      Converged parameter $\alpha_{\kopt}$;
      convergence index $\kopt$;
      relative mean change $\hat c_{iw}$;
      }

  $k \longleftarrow k_0; \quad \hat c_{iw} \longleftarrow \text{inf}$\;
  Initialize $\alphamin$ according to \eqref{tightlowerbound} using $\ell$ if provided, otherwise using $k$ \;
  $\alpha(0) \longleftarrow \argmin{\alpha} U_{k}(\alpha)$ over interval $[\alphamin,1]$ \;
  \While{ $(\hat c_{iw} > \delta$ \bf{and} $k < k_\text{max})$ \bf{or} $(\alpha(i) = \alphamin)$}{
    $i \longleftarrow i + 1; \quad k \longleftarrow k + \Delta_k $\;
    If $\ell$ not provided, update $\alphamin$ according to \eqref{tightlowerbound} using $k$ \;
    $\alpha(i) \longleftarrow \argmin{\alpha} U_{k}(\alpha)$ over interval $[\alphamin,1]$ \; 
    $c(i) = (| \alpha(i) - \alpha(i-1) |) / \alpha(i)$\;
    \If{$i \ge w$}{
    $\hat c_{iw} \longleftarrow \mean(c(i), c(i-1), \dots, c(i - w + 1))$\;
    }
  }

  \Return{$k=\kopt$, $\alpha(i)=\alpha_{\kopt}$, $\hat c_{iw}$}
  % }
  \caption{Truncated UPRE Parameter Estimation\label{influx}}
\end{algorithm}

%%%%%%%%%%%%%%%%% algorithm description
Algorithm~\ref{influx} takes as input a full or truncated SVD as well as a number of required and optional parameters which we now discuss. For large scale problems it is not necessary, and is even discouraged, to compute $\alpha_k$ for all $k \le \kopt$. For moderately or mildly ill-posed problems, and for problems with high signal to noise ratios in which the expected $\kopt$ is likely to be large relative to the problem size, it is recommended to start the algorithm at some $k_0 \ne 1$ and to increment $k$ by some $\Delta_k \ne 1$, yielding the sequence $\{k(i): k_0, k_0 + \Delta_k, k_0 + 2\Delta_k, \dots k_0 + i\Delta_k\}$. The algorithm computes the sequence $\{ \alpha_{k_0}, \alpha_{k_0 + \Delta_k}, \alpha_{k_0 + 2\Delta_k}, \alpha_{k_0 + 3\Delta_k}, \dots \} $, each solving \eqref{argmin} for the given index,  until either $k_0 + i \Delta_k \ge k_\text{max}$ or until $\alpha_k$ has converged, where $k_0$, $\Delta_k$, and $k_\text{max}$ are provided by the user.  For each $k_0 + i\Delta_k$ the relative change in $\alpha$ is computed as $c_i = | \alpha_{k_0 + i\Delta_k} - \alpha_{k_0 + (i-1) \Delta_k} | / \alpha_{k_0 + i\Delta_k}$.
Noting again that $c_i$ is only $\emph{in general}$ decreasing for sufficiently large $i$, it is unwise to determine stopping criteria by directly thresholding on $c_i < \delta$, for some user provided tolerance $\delta$.  It is observed that higher confidence in convergence can be achieved by requiring  $\hat{c}_{iw}<\delta$ where $\hat c_{iw}$ is the mean of multiple $c_i$'s calculated over the window of size $w$, i.e. over $\{c_i, c_{i+1}, \dots, c_{i+w} \}$.  This protects against the possibility of stopping the parameter search too early and prior to the stabilization of $\alpha_k$. This occurs when $c_i < \delta$, while at the same time $c_j \ge \delta$ for some $j > i$. Due to the impact of noise on calculating the parameter $\alpha_k$, if $k$ is not yet sufficiently large so that $\alpha_k$ has not stabilized then the relative changes between successive estimates of $\alpha_k$ may be either extremely small or large. Comparing multiple values of $c_i$ in the form of $\hat c_{iw}$ to $\delta$ enables a broader view of the convergence of $\alpha_k$, and the moving window average smooths out variation in $c_i$. 
\begin{remark}[Parameter $\Delta_k$]
The choice of $\Delta_k$ is influenced by the size of the problem and if known, an estimate for the expected number of terms to be used in the TSVD solution. While choosing $\Delta_k$ large has computational advantages due to a larger step size in the search for $\kopt$, with $\Delta_k$ too large one risks the possibility of Algorithm~\ref{influx} producing a value of $\kopt$ larger than necessary. Solutions with $\kopt$ larger than necessary more closely resemble the full UPRE regularized solution. For the problem sizes considered here $\Delta_k \in \{5, 10, 25\}$ all seemed to work well. 
\end{remark}
\begin{remark}[Parameter $w$]
The choice of $w$ has a similar effect as $\Delta_k$. Choosing $w$ large will delay the termination criteria . Parameters $\Delta_k$ and $w$ interact in the sense that they together determine the set $\{c_i, c_{i+1}, \dots, c_{i + w} \}$ whose mean is compared to $\delta$ in determining convergence. The choice of $w$ determines how many values are being averaged, while $w$ and $\Delta_k$ determine the minimum and maximum $k$ of the moving window over which $\alpha_k$ is tested for convergence. Choosing $w \in \{5, 10, 25, 50\}$ worked well for the problems considered here.
\end{remark}
\begin{remark}[Parameter $\delta$]
Algorithm~\ref{influx} is sensitive to $\delta$ and we recommend choosing $\delta \in [1\mathrm{e}{-5}, 1\mathrm{e}{-3}]$. In our experiments $\delta > 1\mathrm{e}{-3}$ terminated the algorithm prior to convergence resulting in over smoothed solutions due to an underestimate of $\kopt$, while $\delta < 1\mathrm{e}{-5}$ produced $\kopt$ far greater than necessary. 
\end{remark}

To summarize, the required input to the proposed algorithm is a full or  truncated SVD, a starting index $k_0$, a step size between successive estimates $\Delta_k$, an upper-bound $k_\text{max}$ dependent on the severity of the problem and the noise level, a tolerance $\delta$, and a width $w$ over which the moving average of relative changes in successive estimates of $\alpha$ is computed.

The results of Theorem~\ref{thm:minbnd} are incorporated into Algorithm~\ref{influx} with the inclusion of an optional parameter $\ell$ specifying an estimate for the index at which noise dominates the coefficients. If a Picard plot is available $\ell$ can be estimated visually, otherwise an approach relying on Picard parameter estimates similar to that used by \cite{Taroudaki} and \cite{Levin} can be used. If an estimate for $\ell$ is available, $\alphamin$ is calculated according to \eqref{tightlowerbound}, and $\alpha_k$ is found using $\alphamin = \sigma_{\ell+1} / \sqrt{1- \sigma_{\ell+1}^2}$ and $\alphamax=1$ in \eqref{argmin}. Otherwise,  the bound  $\sigma_{k+1} / \sqrt{1- \sigma_{k+1}^2}$ is used in \eqref{argmin}. In either case if the lower bound is achieved then the theory indicates that noise has not yet dominated  and the algorithm is allowed to continue. Thus, in the case where $k < k_\text{max}$, necessary conditions for the termination of Algorithm~\ref{influx} are $\hat c_{iw} < \delta$ and $\alpha_k$ should be greater than the specified $ \alphamin$.

%%%%%%%%%%%%%%%%%% Results of algorithm
\subsection{Verification of the Algorithm and Theory}\label{sec:algresults}
We now present the evaluation of Algorithm~\ref{influx} on a 2D test problem using the IR Tools package described in \cite{GaHaNa:IR}. We report the results applying a Gaussian blur to test problem \texttt{Satellite} of size $256\times256$ using \texttt{PRblur}, with  medium blur.  We considered noise levels of $5\%$, $10\%$, and $25\%$, with $100$ noise instances generated for each noise level. The IR Tools function \texttt{PRnoise} was used to generate noise, where the noise level is defined as as $\| \mb \eta \|_2 / \| \mb b \|$. A moving window of size $w=5$  in computing $\hat c_{iw}$ with relative tolerance of $\delta = 1\mathrm{e}{-3}$ was found to work well for each noise level, but may need to be adapted to the severity of the ill-posedness of the problem. Recorded in each run are the converged $\alpha_{\kopt}$, the size of the TSVD subspace $\kopt$ to be used, and the relative reconstruction error (RRE). RRE is defined as $ \| x_\text{true} - x_{\kopt} \|_2 / \| x_\text{true} \|_2 $ where $x_{\kopt}$ is the filtered, $\kopt$-truncated TSVD solution obtained by using $\alpha_{\kopt}$ as the  regularization parameter.

\begin{figure}[!htb]
    \begin{center}
{\includegraphics[width=.33\textwidth]{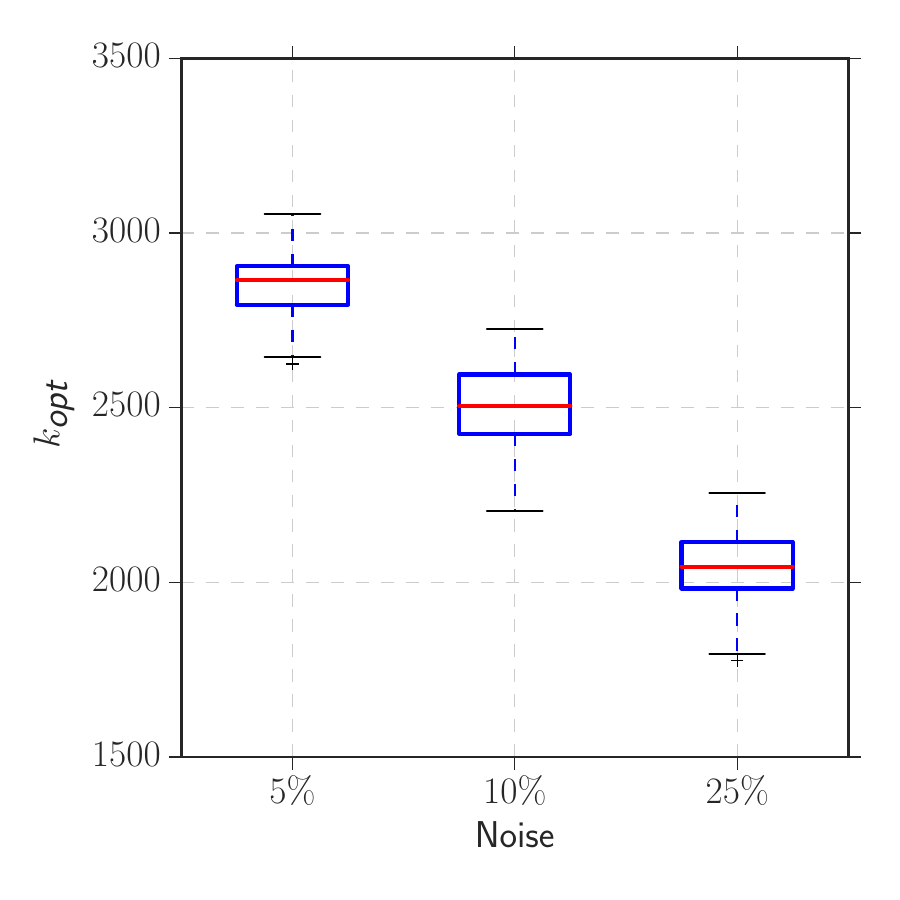}}
    \end{center}
    \caption{Box plots showing the index $\kopt$ produced by Algorithm~\ref{influx} for problem \texttt{Satellite} computed from $100$ runs for  noise levels $5\%$, $10\%$, and $25\%$. The number of terms $k$ in the TSVD that provide useful information decreases as the noise level increases.\label{fig:k_boxplot}}
\end{figure}

% PARAMETER BOXPLOTS
\begin{figure}[!htb]
  \begin{center}
    \subfloat[Noise level $= 5\%$ ]{\includegraphics[width=.33\textwidth]{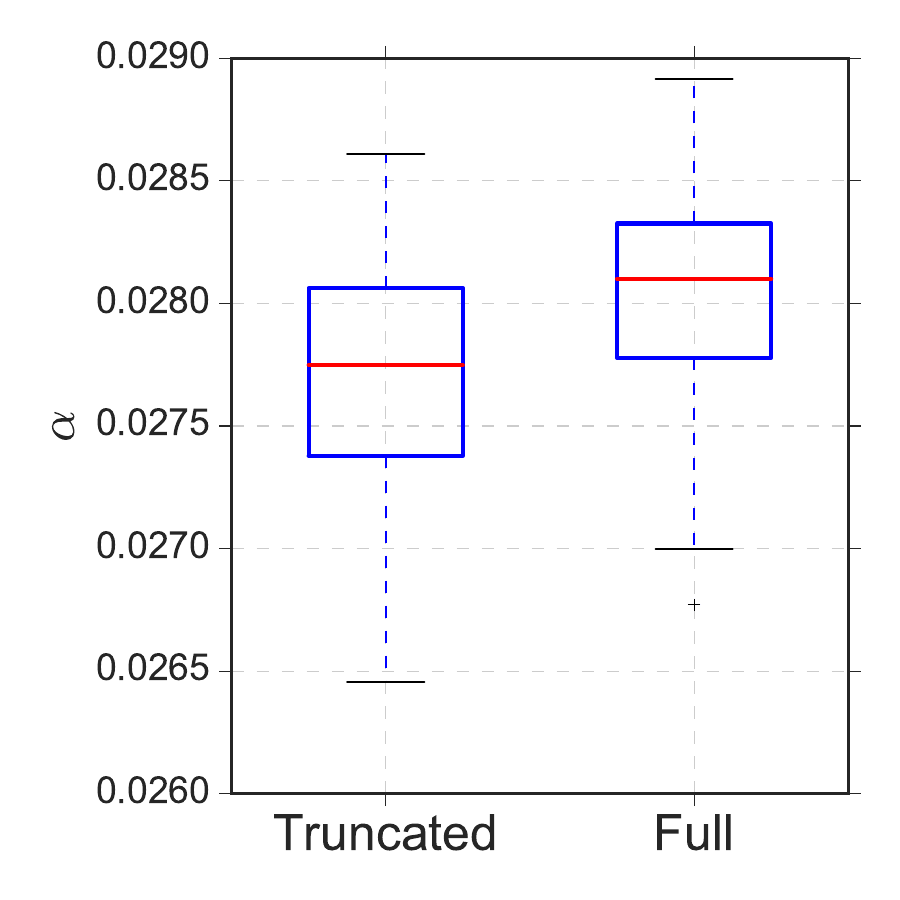}}
    \subfloat[Noise level $= 10\%$ ]{\includegraphics[width=.33\textwidth]{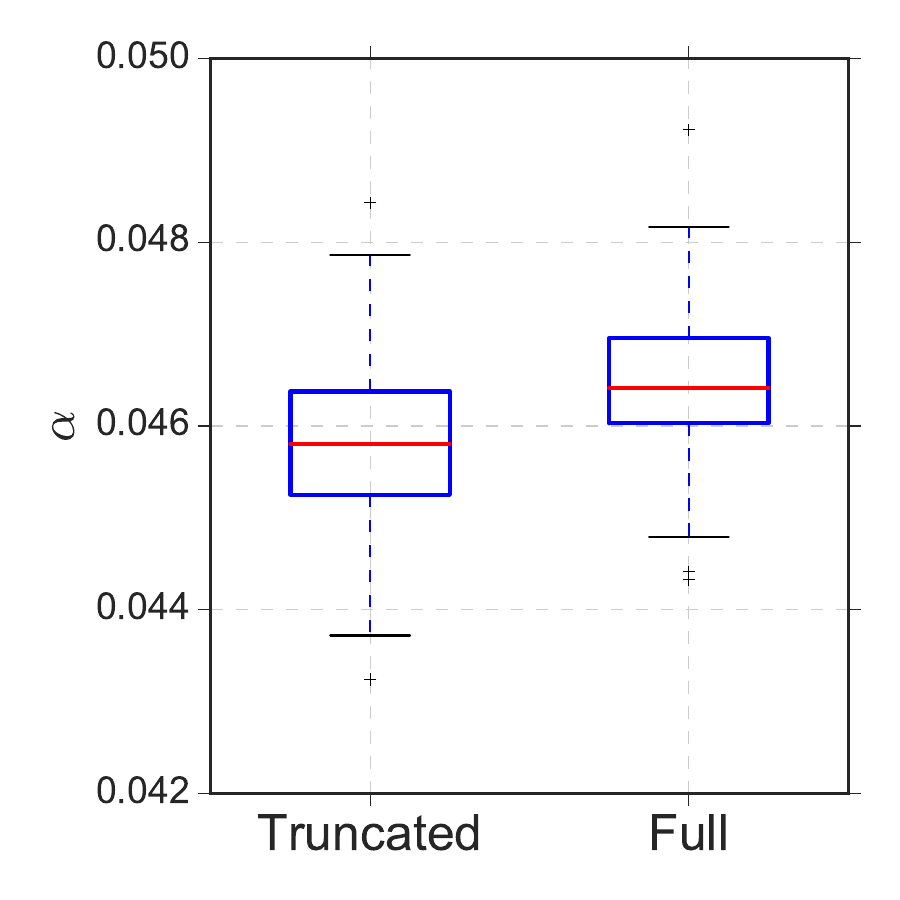}}
    \subfloat[Noise level $= 25\%$]{\includegraphics[width=.33\textwidth]{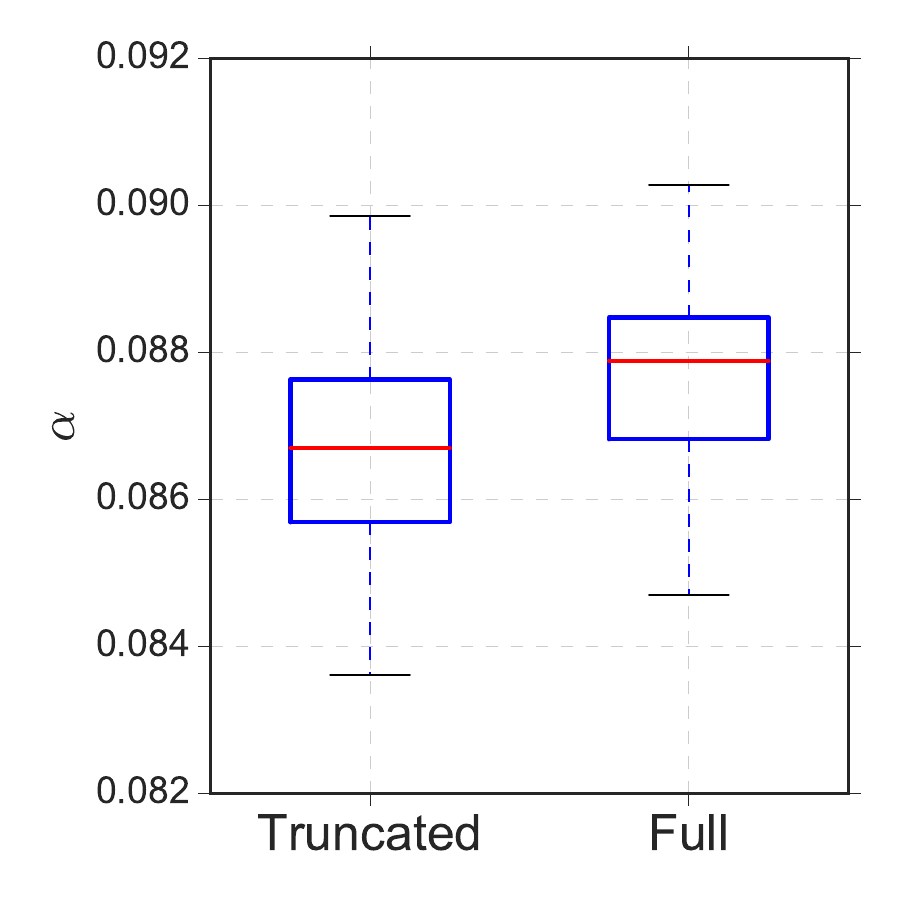}}
  \end{center}
  \caption{Box plots comparing parameter estimates $\alpha_{\kopt}$ with $ \alpha_r$ for problem \texttt{Satellite} computed from $100$ runs for noise levels $5\%$, $10\%$, and $25\%$. For each noise level, the estimate $\alpha_{\kopt}$ produced by Algorithm~\ref{influx} is generally less than $\alpha_r$, demonstrating that by including more terms in the TSVD, $k > \kopt$, greater regularization is required.  Note that the limits on the $y-$axes vary across subplots to better visualize the parameter distributions across noise levels.\label{fig:parameter_boxplots}}
\end{figure}

% PARAMETER LINE PLOTS
\begin{figure}[!htb]
  \begin{center}
    \subfloat[Noise level $= 5\%$ ]{\includegraphics[width=.33\textwidth]{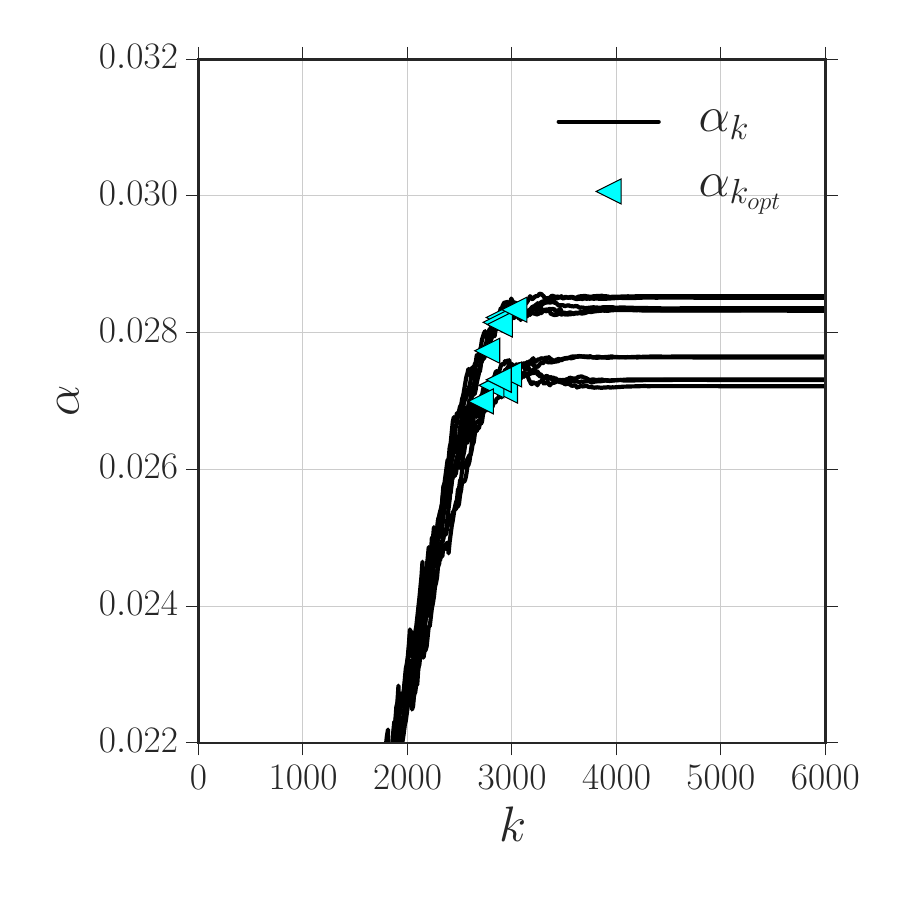}}
    \subfloat[Noise level $= 10\%$ ]{\includegraphics[width=.33\textwidth]{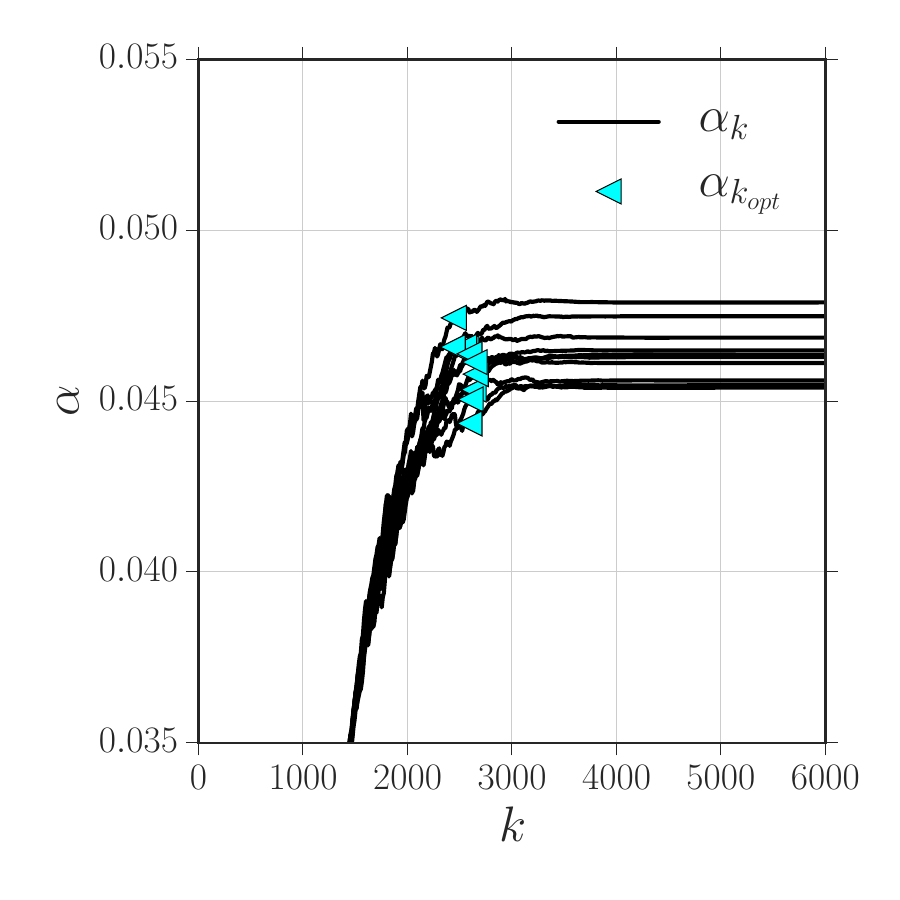}}
    \subfloat[Noise level $= 25\%$ ]{\includegraphics[width=.33\textwidth]{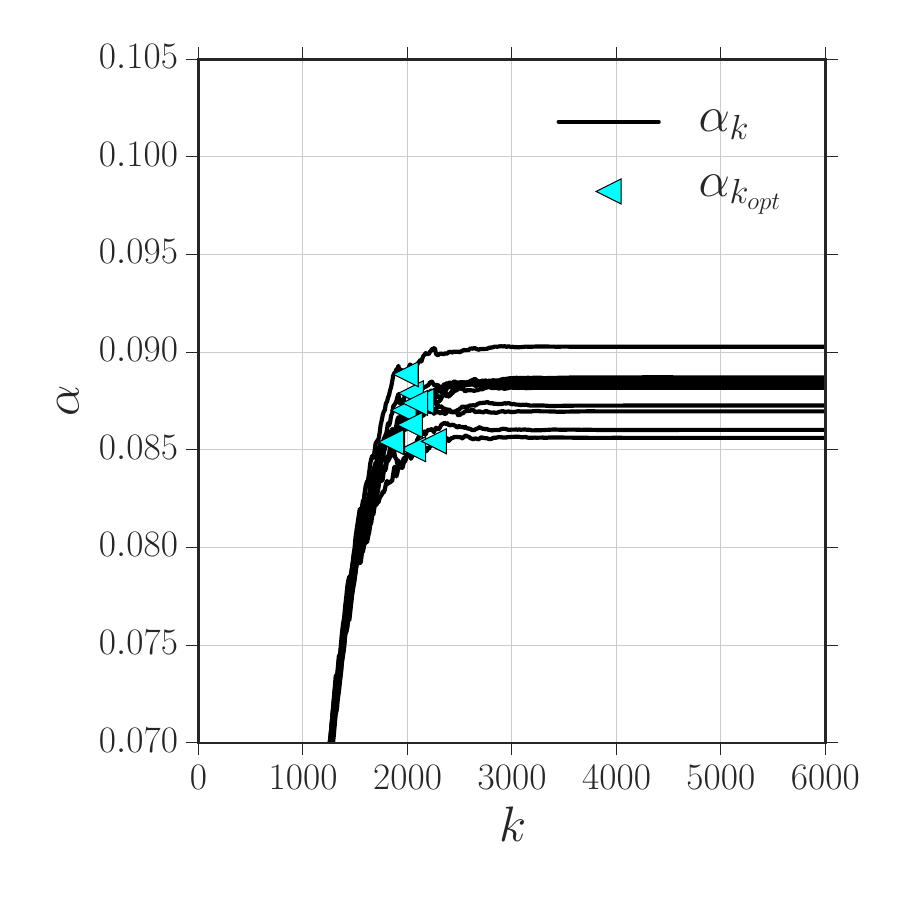}}   
  \end{center}
  \caption{Line plots showing the calculated estimates for $\{\alpha_k\}$ with increasing number of terms $k$ in the TSVD. The results are given for problem \texttt{Satellite} for noise levels $5\%$, $10\%$, and $25\%$, for $10$ random noise instances at the specified noise level. The resulting point $(\kopt, \alpha_{\kopt})$ produced by Algorithm~\ref{influx} is displayed as a cyan triangle. Note that the limits on the $y-$axes vary across subplots to better visualize the convergence across noise levels. \label{fig:parameter_lines}}
\end{figure}

Figure ~\ref{fig:k_boxplot} is a box plot\footnote{A box plot is a visual representation of summary statistics for a given sample. Horizontal lines of each plotted box represent the $75\%$, $50\%$ (median), and $25\%$ quantiles, with outliers plotted as individual crosses or points.} showing the spread of $\kopt$ values for the $100$ noise instances run for each noise level, where in each case $\kopt \ll r = 65536$. Figure~\ref{fig:parameter_boxplots} is a box plot comparing the $\alpha_{\kopt}$ returned by the algorithm, and $ \alpha_r$ obtained by minimizing the UPRE on the full space. These figures together reaffirm that the optimal regularization parameter found by UPRE is largely determined by a relatively small number of terms in the TSVD, and less impacted by the tail of the coefficients dominated by noise. The estimation of $\{ \alpha_k \}$ with increasing number of terms in the TSVD is depicted in Figure~\ref{fig:parameter_lines} for the first $10$ runs of each noise level, where the point of convergence  $(\kopt, \alpha_{\kopt})$ is represented as a cyan triangle. It should be noted that  the estimated lower bound $\alphamin$ was not used, and $\alpha$ was minimized over the interval $(0,1)$ using \texttt{fminbnd} (\texttt{fminbound} is used for the Python implementation). A tolerance of $\delta = 1\mathrm{e}{-3}$ was found to produce a value for $\alpha_{\kopt}$ just prior to the point where  $\{ \alpha_k \}$ began to stabilize. A smaller $\delta$ will necessarily increase $\kopt$, but with negligible changes in $\alpha_{\kopt}$. In these simulations averaged over all $100$ runs, $\alpha_{\kopt}$ was within $1.22\%$, $1.47\%$, and $1.17\%$ of $ \alpha_r$ for noise levels $5\%$, $10\%$, and $25\%$ respectively using fewer than $5\%$ of the SVD components.

In terms of RRE, the solution obtained using the truncated  UPRE  and a subspace of size $\kopt$ with  parameter $\alpha_{\kopt}$ generated by Algorithm~\ref{influx} generally provided a better solution than obtained using the full UPRE for each noise level. Figures~\ref{fig:rre_boxplots} and ~\ref{fig:rre_histograms} show box plots and histograms respectively of the RRE comparing the regularized TSVD and the full UPRE solution.  Over all noise levels, the median and mean reconstruction error of $100$ noise instances is lower in the regularized TSVD solution. Similar to the Picard parameter approaches of \cite{Taroudaki}, Algorithm~\ref{influx}  identifies an index $\kopt$ for which coefficients $s_k$ are dominated by noise for $k > \kopt$. Our approach, however, does not rely on performing statistical tests on the coefficients, but instead examines the stabilization of $\alpha_k$ as $k$ increases. Once $\alpha_k$ has stabilized, adding additional noise dominated terms in the solution delivers no benefit.  Furthermore, if a TSVD with $\kmax$ terms has been calculated, then either $\alpha_k$ converges for $k<\kmax$ or we know that the optimal choice $\kopt$ is greater than $\kmax$, and that $\hat{c}_{iw}$ provides some estimate for whether $\kopt>>\kmax$ or whether the given TSVD can be assumed to be sufficient in providing a good estimate for the solution $\mb x$.

%%%%%%%%%%%%%%%%%%%%% RRE Figures
% RRE BOXPLOTS
\begin{figure}[!htb]
  \begin{center}
    \subfloat[Noise level $= 5\%$ ]{\includegraphics[width=.33\textwidth]{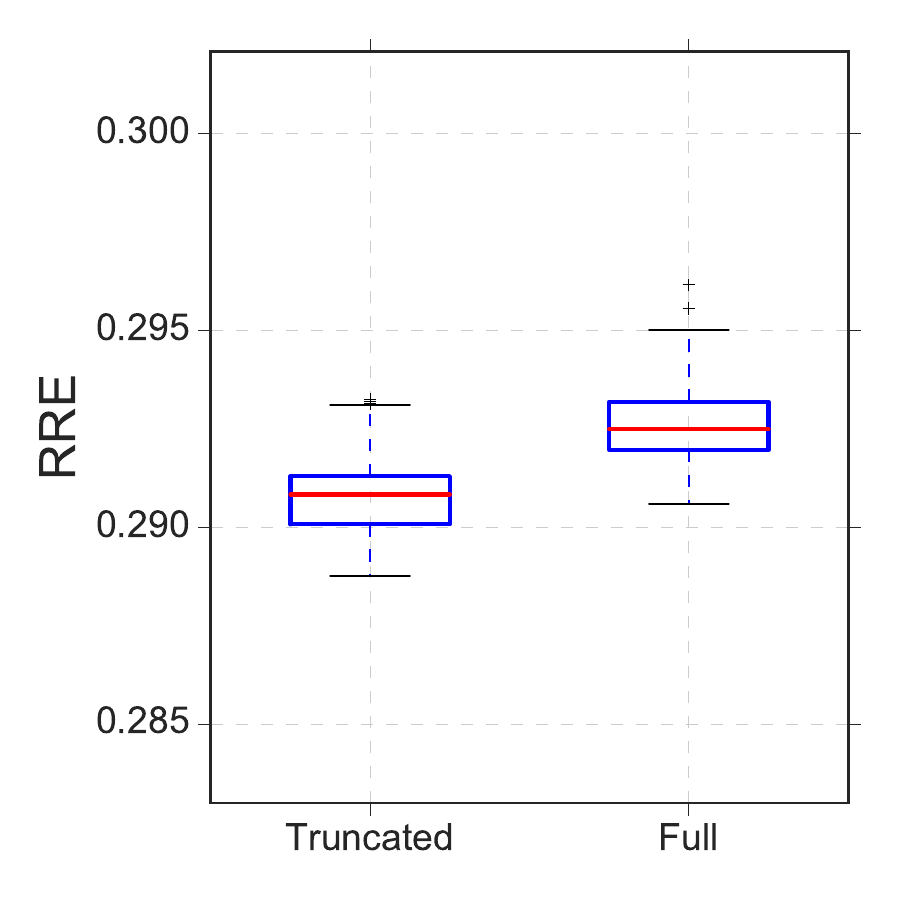}}
    \subfloat[Noise level $= 10\% $]{\includegraphics[width=.33\textwidth]{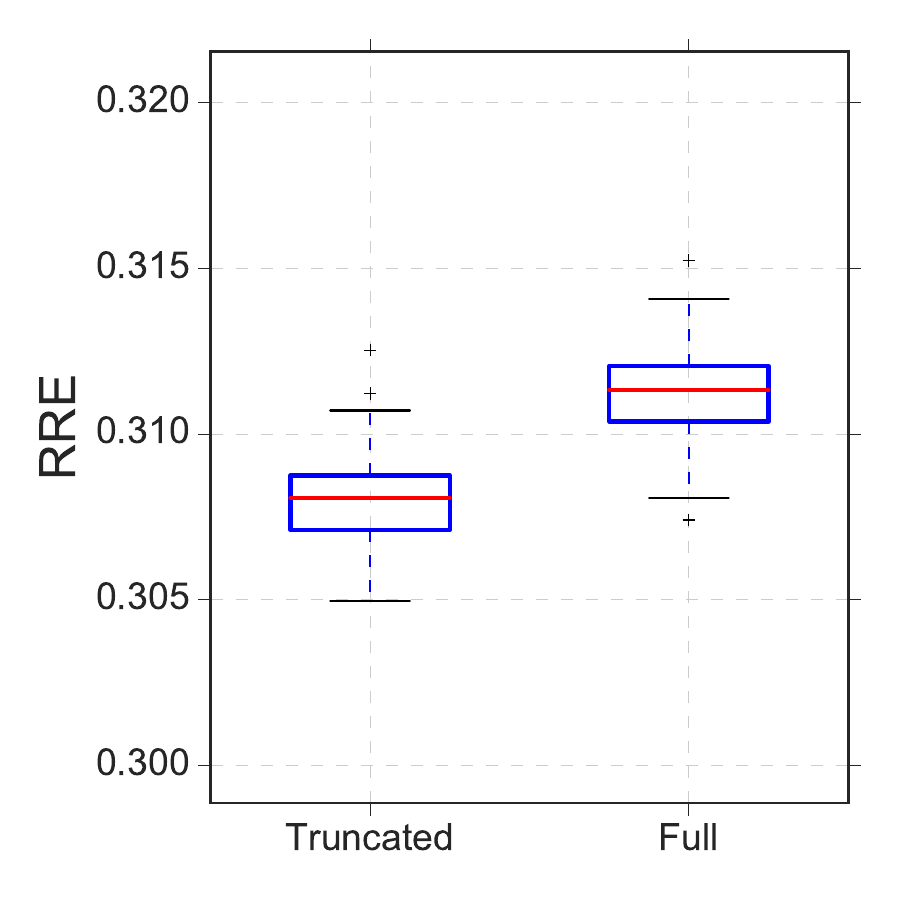}}
    \subfloat[Noise level $= 25\%$]{\includegraphics[width=.33\textwidth]{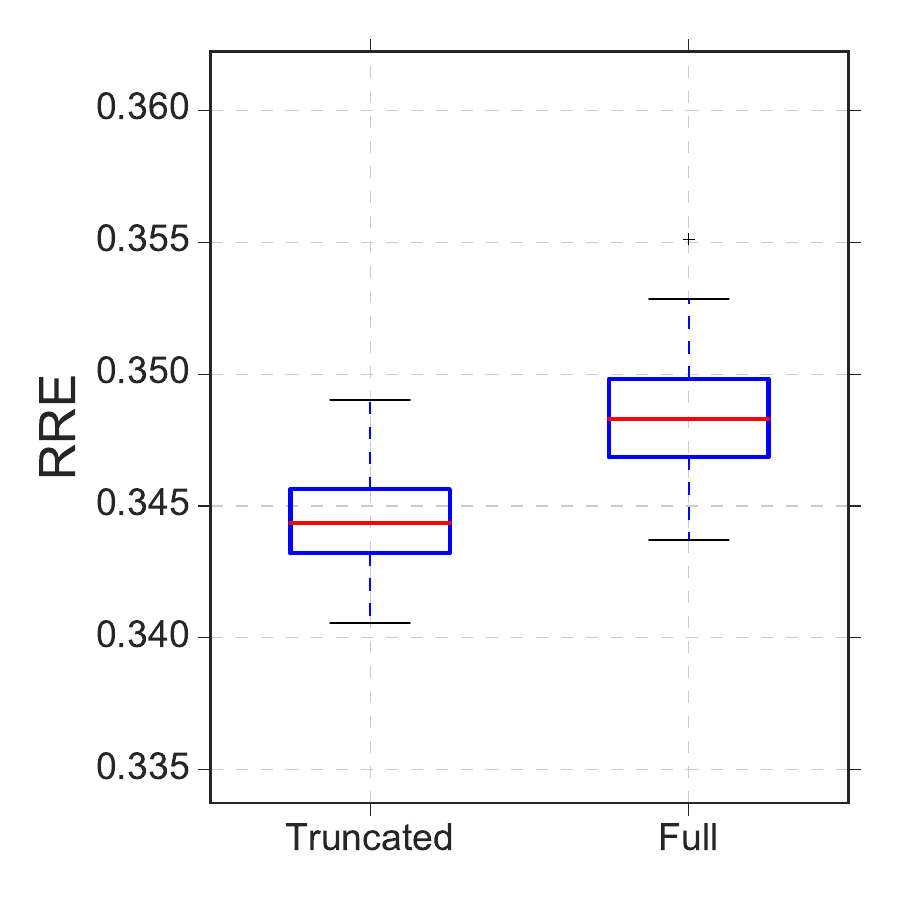}}
  \end{center}
  \caption{Box plots of RRE comparing solutions using truncated UPRE with parameter $\alpha_{\kopt}$ and solutions using full UPRE with parameter $ \alpha_r$ for problem \texttt{Satellite} computed from $100$ runs for noise levels $5\%$, $10\%$, and $25\%$. Regularization parameter $\alpha_{kopt}$ obtained by UPRE on a TSVD generally has lower error, as evident from Truncated UPRE plots being vertically shifted downwards relative to full UPRE boxplots. Note that the limits on the $y-$axes vary across subplots to better visualize the spread of the distributions across noise levels. \label{fig:rre_boxplots}}
\end{figure}

% RRE HISTOGRAMS
\begin{figure}[!htb]
  \begin{center}
    \subfloat[Noise level $= 5\%$ ]{\includegraphics[width=.33\textwidth]{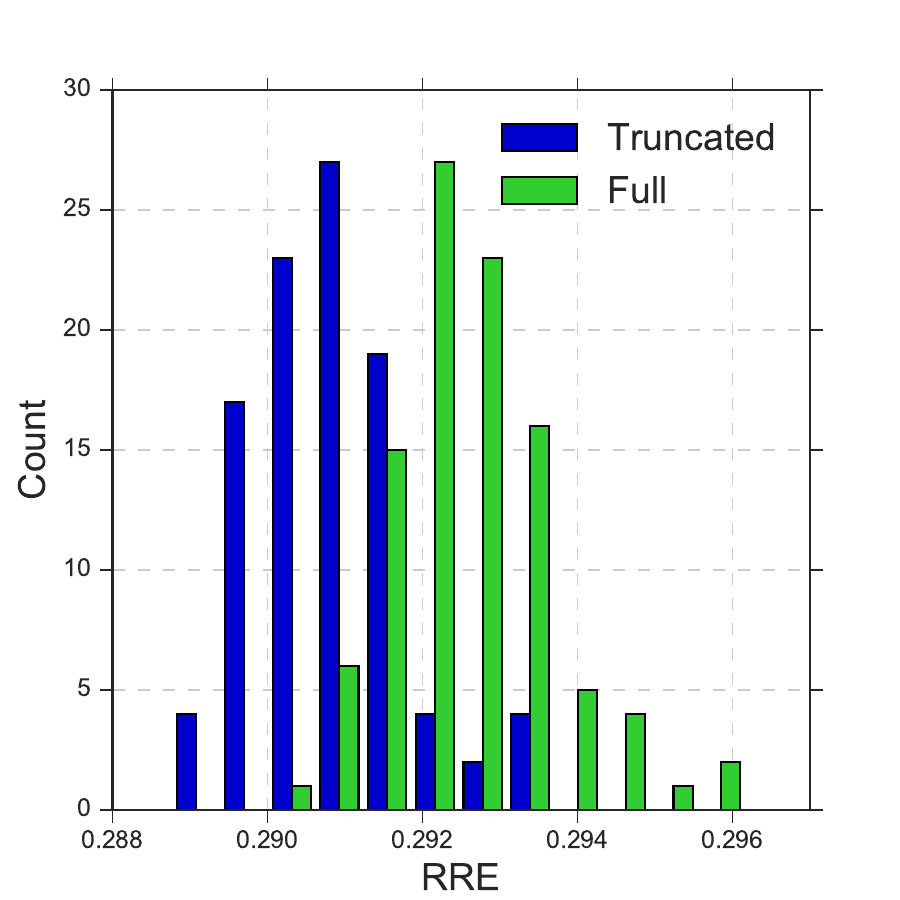}}
    \subfloat[Noise level $= 10\%$ ]{\includegraphics[width=.33\textwidth]{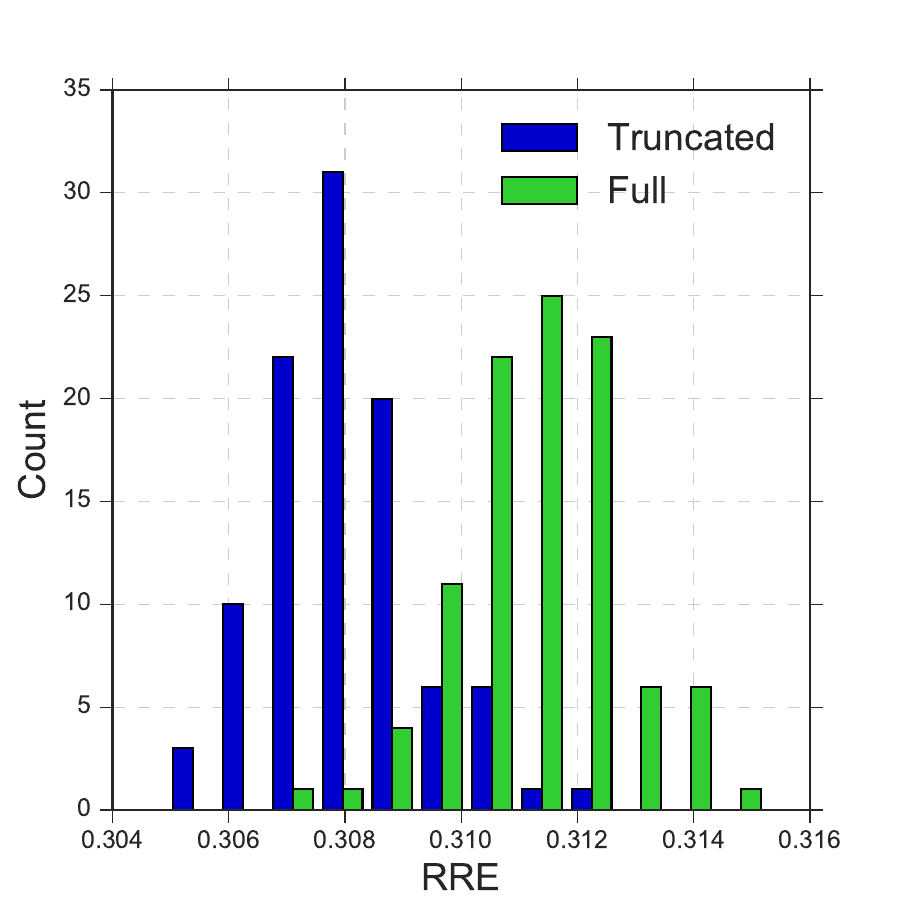}}
    \subfloat[Noise level $= 25\%$ ]{\includegraphics[width=.33\textwidth]{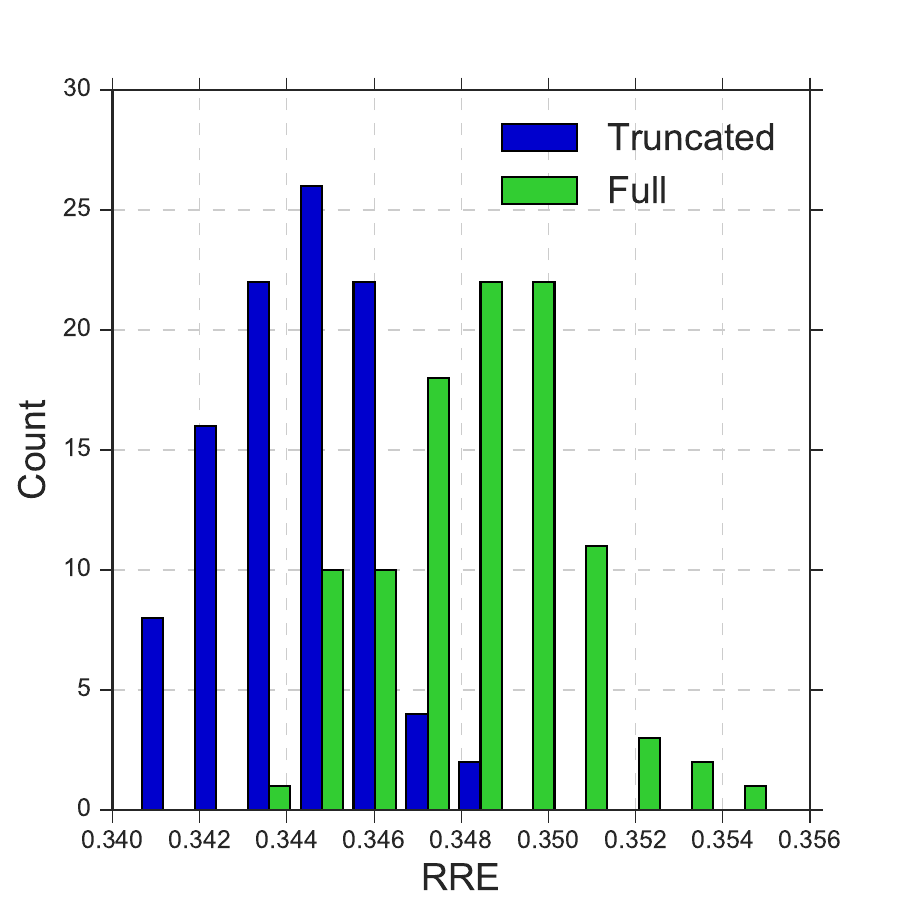}}
  \end{center}
  \caption{Histograms of RRE comparing solutions using $\alpha_{\kopt}$ and solutions using $ \alpha_r$ for problem \texttt{Satellite} computed from $100$ runs for each noise level $5\%$, $10\%$, and $25\%$. Regularization parameter $\alpha_{\kopt}$ obtained by UPRE on a TSVD generally has lower error, as evident from the truncated histograms having peaks shifted to the left relative to the full UPRE.\label{fig:rre_histograms}}
\end{figure}

In these simulations  $\ell$ is not known precisely but was estimated by visual inspection of the Picard coefficients, as well as by comparing the distributions of the noise contaminated and noise free coefficients. This approach for estimating $\ell$ is not possible in general as the noise free coefficients are unknown in practice, but this method of estimating $\ell$ was employed for the purpose of validating the results of Theorem~\ref{thm:minbnd}. An estimate for the lower bound $\alphamin$ obtained from \eqref{tightlowerbound} is depicted as the red dashed curve in Figure~\ref{fig:lower_bound}, with $\{\alpha_k\}$ the solid black line. It can be seen that $\alphamin$ serves as a tight lower bound for the converged parameter $\alpha_{\kopt}$, and  the lower bound   $\sigma_{k+1} / \sqrt{1 - \sigma_{k+1}^2}$ can be used effectively in cases where an estimate of $\ell$ is not available.

In addition to test image \texttt{Satellite} with a medium Gaussian blur applied, we also applied Algorithm~\ref{influx} with the same parameters to test image \texttt{HST} with both mild and severe Gaussian blurring. The results, summarized in Figures~\ref{fig:k_boxplot_hst_mild} 
%\ref{fig:additional_cases_1}, and 
- \ref{fig:additional_cases_2} are consistent with the results for test case \texttt{Satellite}.

In summary, given a TSVD or SVD, an optional estimate of $\ell$, and suitable parameters determined by the ill-posedness of the problem, Algorithm~\ref{influx} is able to effectively determine a regularization parameter $\alpha_{\kopt}$ obtained by UPRE minimization over the TSVD subspace of size $\kopt$, such that the regularized truncated solution $\bf x$ has consistently lower RRE than the full UPRE solution.

%%%%%%%%%%%%%%%%%%%%% Results incorporating ell
% RESULTS INCLUDING LOWER BOUND AS DETERMINED BY ELL
\begin{figure}[!htb]
  \begin{center}
    \subfloat[Noise $=5\%$]{\includegraphics[width=.33\textwidth]{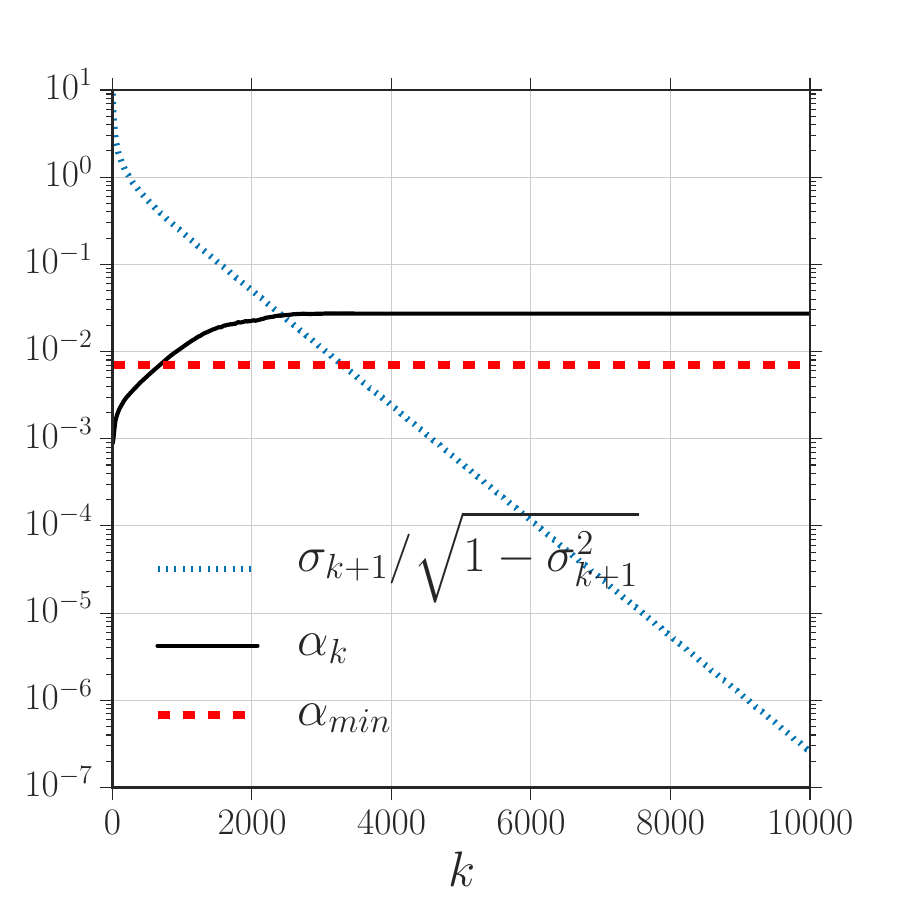}}
    \subfloat[Noise $=10\%$]{\includegraphics[width=.33\textwidth]{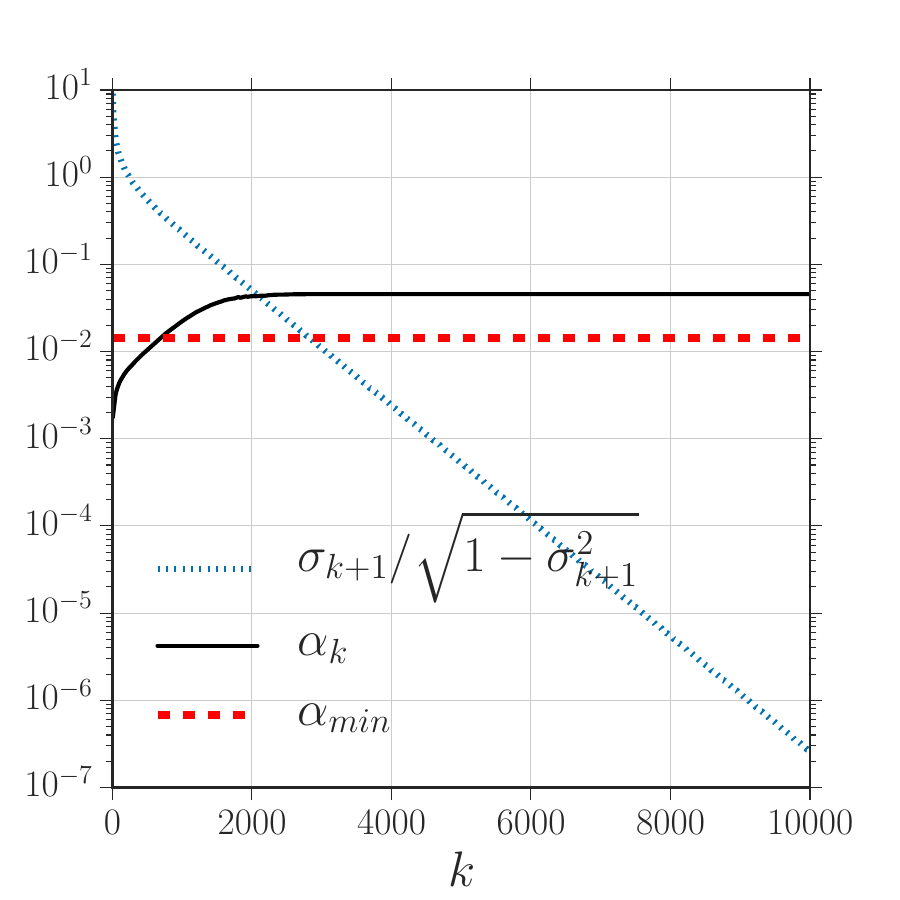}}
    \subfloat[Noise $=25\%$]{\includegraphics[width=.33\textwidth]{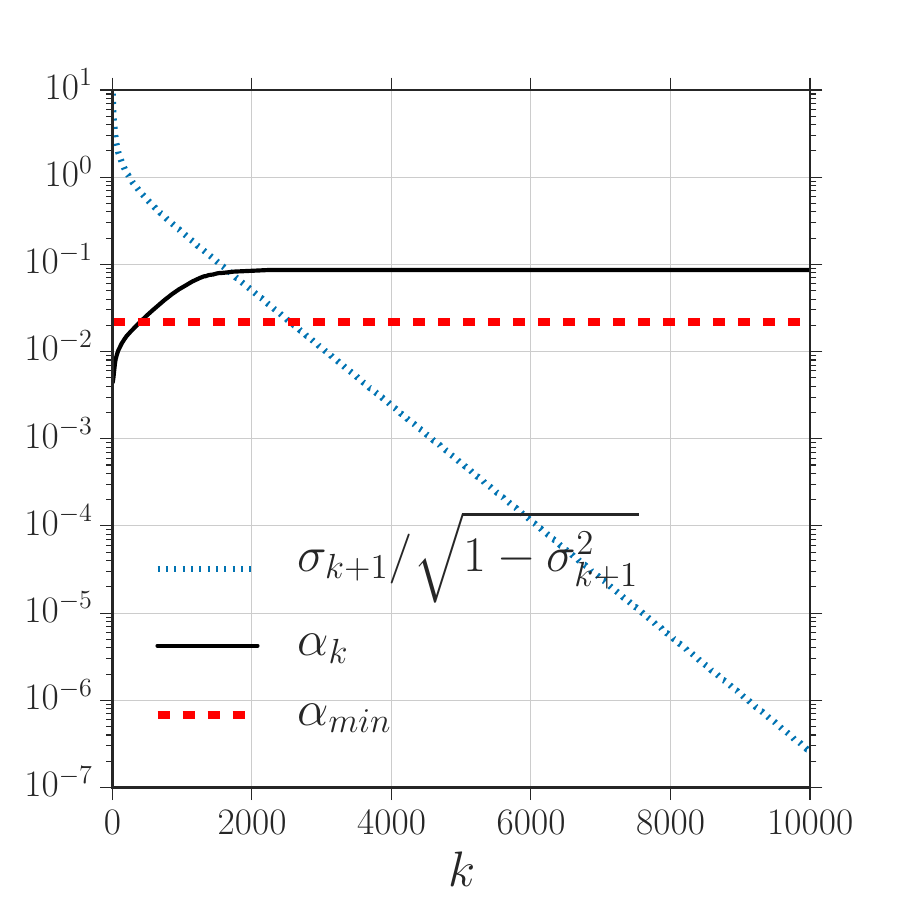}}
  \end{center}
   \caption{Line plots showing the convergence of $\{\alpha_k \}$ for problem \texttt{Satellite} for noise levels $5\%$, $10\%$, and $25\%$. In each subplot, $ \alpha_r$ is plotted as  a solid  black line  for $10$ random noise instances at the specified noise level. The dotted blue curve represents the lower bound in \eqref{tightlowerbound} as a function of $k$, with the red dashed line representing the lower bound according to Theorem~\ref{thm:minbnd} and dependent on $\ell$ for a single run. \label{fig:lower_bound}}
\end{figure}

% ADDITIONAL TEST CASES
\begin{figure}[!htb]
    \begin{center}
    	\subfloat[\texttt{HST} severe blurring]{\includegraphics[width=.33\textwidth]{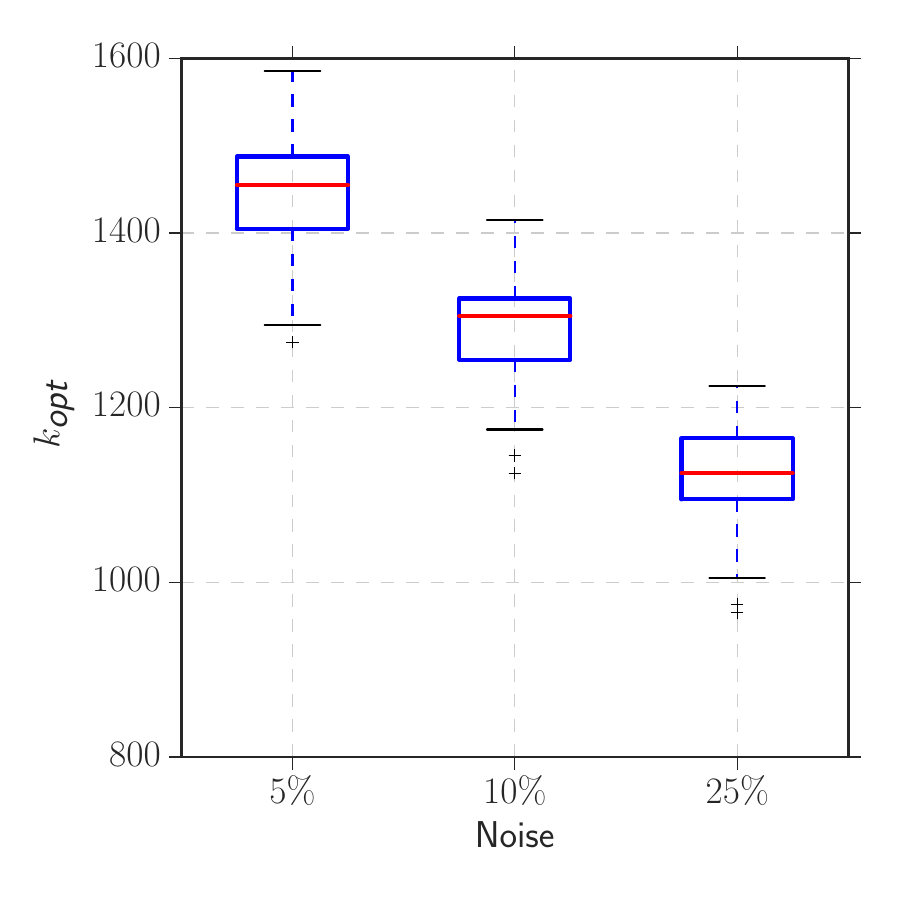}}
	    \subfloat[\texttt{HST} mild blurring ]{\includegraphics[width=.33\textwidth]{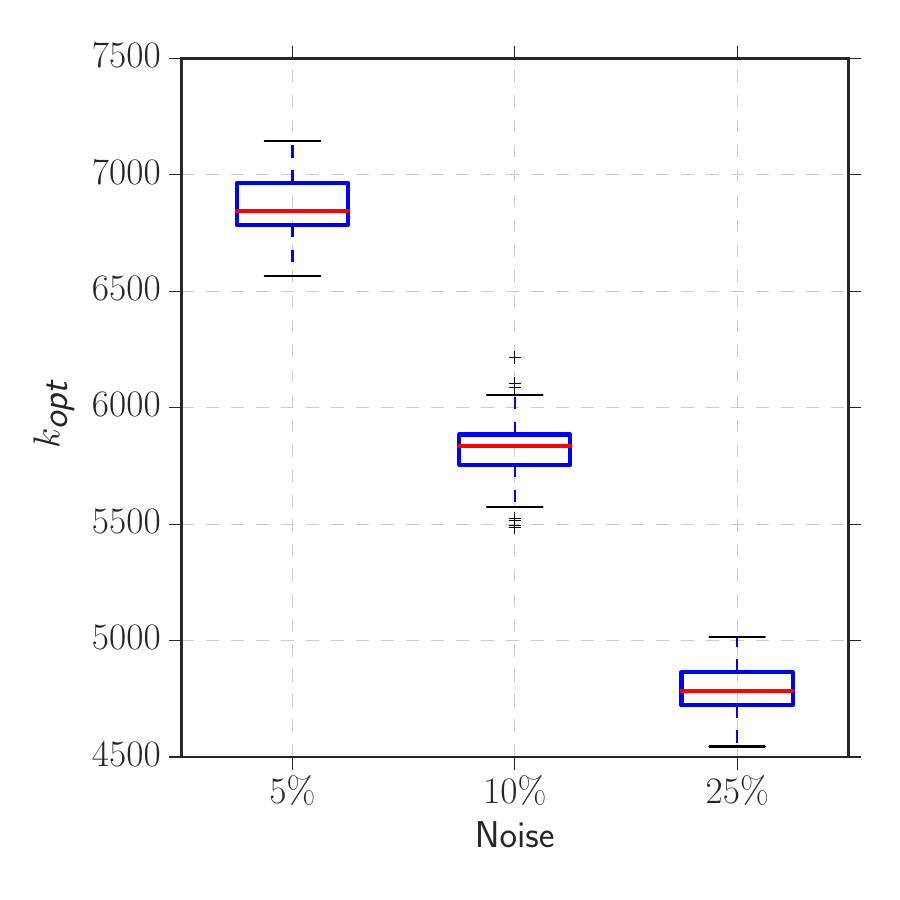}}
    \end{center}
    \caption{Box plots showing the index $\kopt$ produced by Algorithm~\ref{influx} for problem \texttt{HST} computed from $100$ runs for  noise levels $5\%$, $10\%$, and $25\%$. The number of terms $k$ in the TSVD that provide useful information decreases as the noise level increases.\label{fig:k_boxplot_hst_mild}}
\end{figure}

% ADDITIONAL TEST CASES
\begin{figure}[!htb]
  \begin{center}
    \subfloat[Noise level $= 5\%$ ]{\includegraphics[width=.33\textwidth]{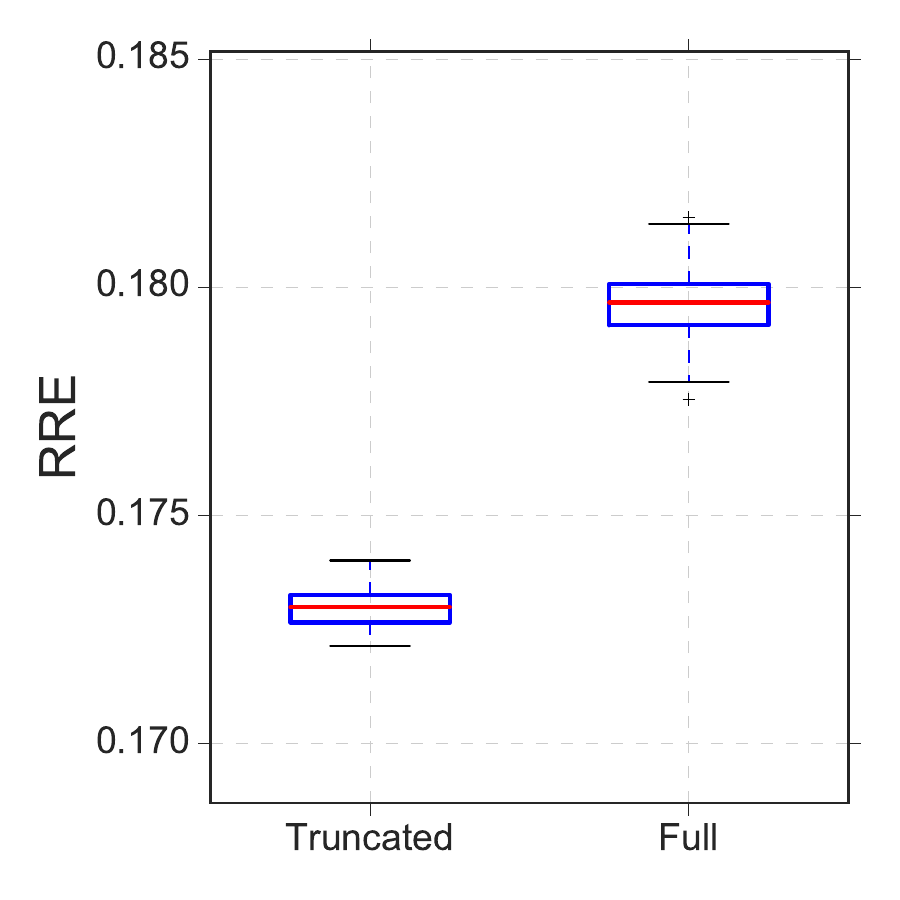}}
    \subfloat[Noise level $= 10\%$ ]{\includegraphics[width=.33\textwidth]{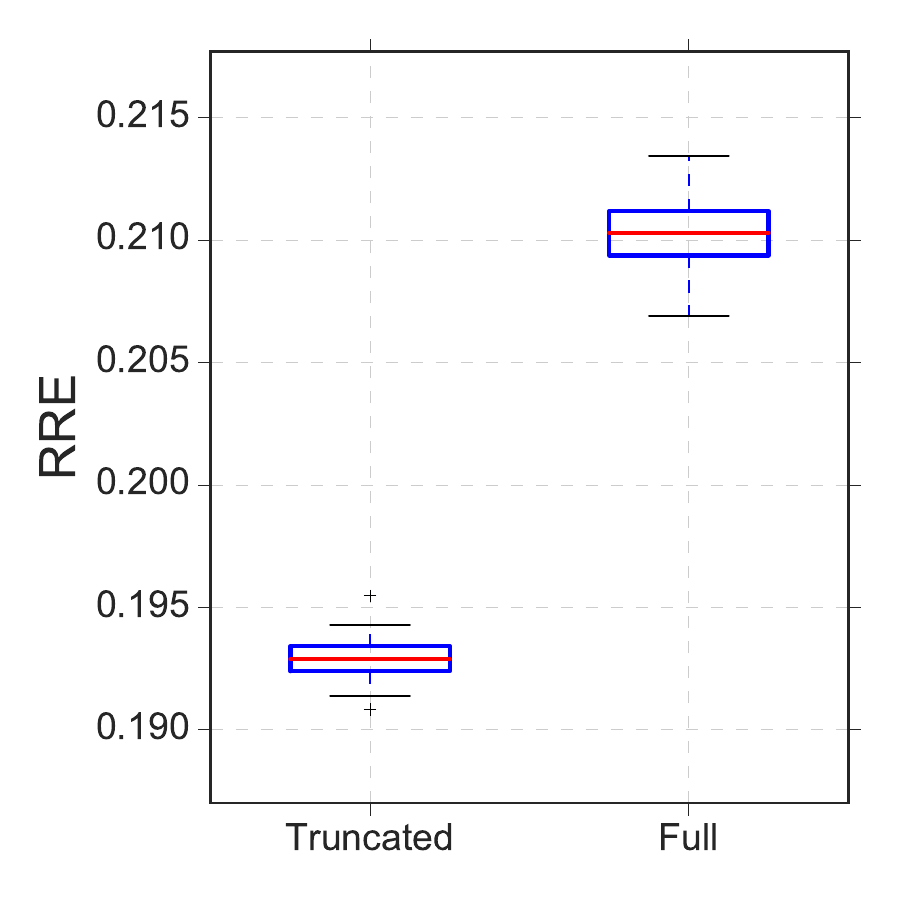}}
    \subfloat[Noise level $= 25\%$ ]{\includegraphics[width=.33\textwidth]{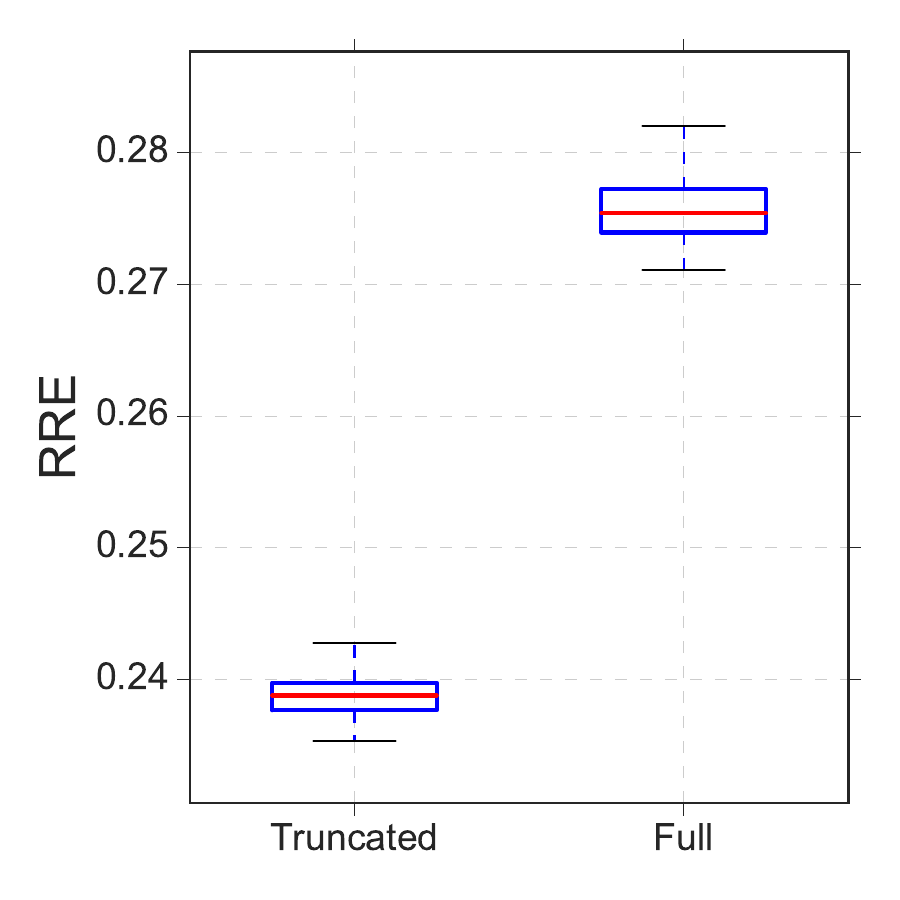}}
  \end{center}
  \caption{Box plots of RRE comparing solutions using truncated UPRE with parameter $\alpha_{\kopt}$ and solutions using full UPRE with parameter $ \alpha_r$ for problem \texttt{HST} with mild blur computed from $100$ runs for noise levels $5\%$, $10\%$, and $25\%$. Regularization parameter $\alpha_{\kopt}$ obtained by UPRE on a TSVD has consistent lower error, as evident from Truncated UPRE plots being vertically shifted downwards relative to full UPRE boxplots. Note that the limits on the $y-$axes vary across subplots to better visualize the spread of the distributions across noise levels.}
%\label{fig:additional_cases_1}}
\end{figure}

% ADDITIONAL TEST CASES
\begin{figure}[!htb]
  \begin{center}
    \subfloat[Noise level $= 5\%$ ]{\includegraphics[width=.33\textwidth]{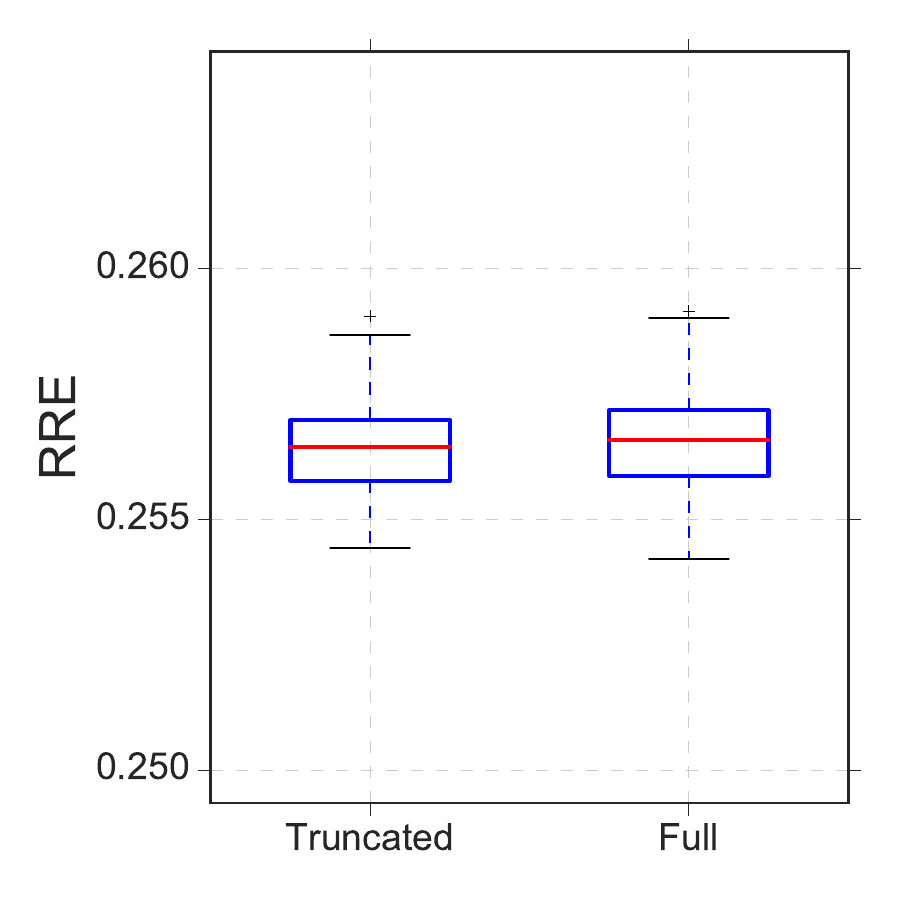}}
    \subfloat[Noise level $= 10\%$ ]{\includegraphics[width=.33\textwidth]{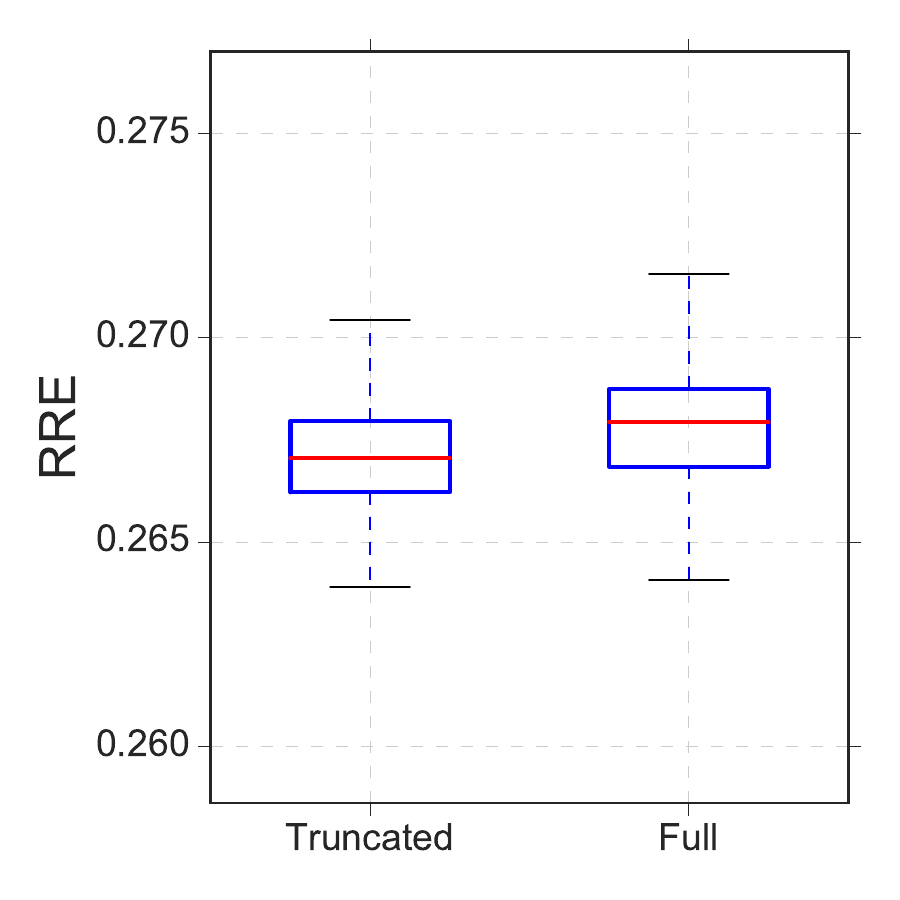}}
    \subfloat[Noise level $= 25\%$ ]{\includegraphics[width=.33\textwidth]{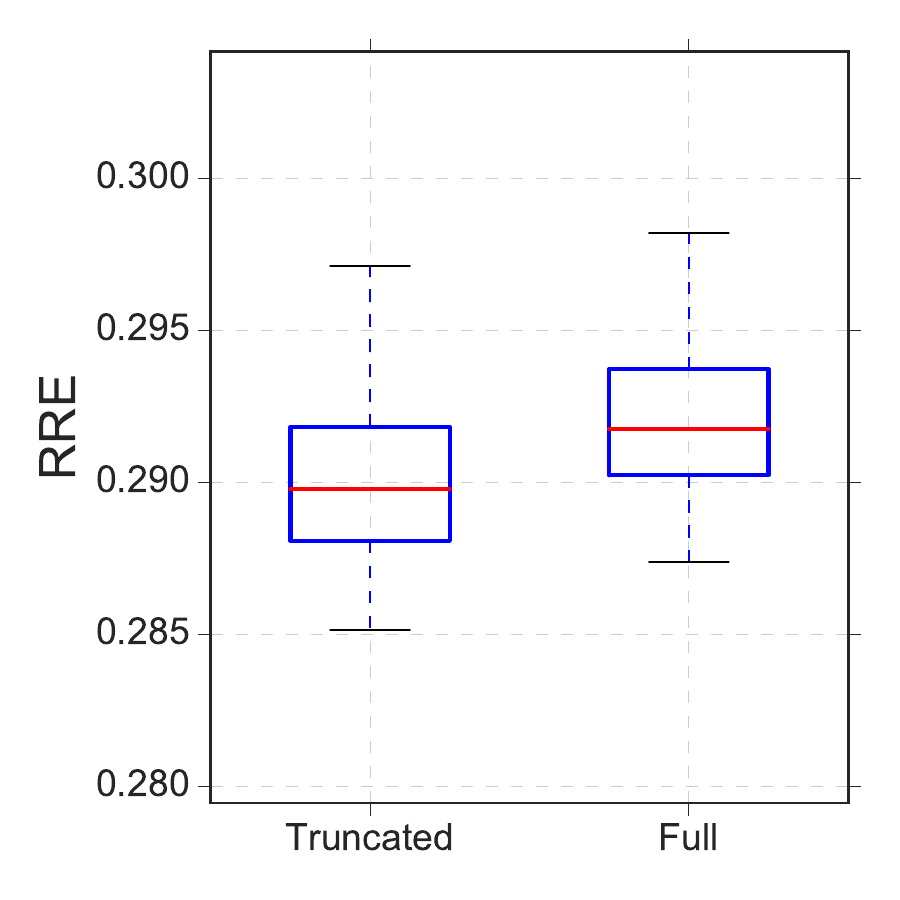}}
  \end{center}
  \caption{Box plots of RRE comparing solutions using truncated UPRE with parameter $\alpha_{\kopt}$ and solutions using full UPRE with parameter $ \alpha_r$ for problem \texttt{HST} with severe blur computed from $100$ runs for noise levels $5\%$, $10\%$, and $25\%$. Regularization parameter $\alpha_{\kopt}$ obtained by UPRE on a TSVD generally has lower error for noise levels $10\%$ and $25\%$, with comparable error for noise level $5\%$. Note that the limits on the $y-$axes vary across subplots to better visualize the spread of the distributions across noise levels.
\label{fig:additional_cases_2}}
\end{figure}

\section{Conclusions\label{sec:conclusions}}
We have demonstrated that the regularization parameter obtained using the UPRE estimator converges with increasing number of terms  used from the TSVD for the solution. For a  severely ill-posed problem the  convergence occurs very quickly and is independent of the size of the problem due to the fast contamination of data coefficients by practical levels of noise. Practically-relevant problems are often, however, only moderately or mildly ill-posed, e. g. \cite{ChNaOl:08,walnut1,vatan:2014,VRA:2017}, and  it is therefore important to accurately and efficiently find both $\kopt$ and $\alpha_{\kopt}$. 

Theoretical results have been presented that demonstrate the convergence of  the regularization parameter $\alpha_k$ with $k$, increasing from below to $   \alpha_{\kopt}\le \alpha_r$, the optimal value for the full SVD. The posterior covariance thus decreases with $k$, leveling at approximately  $\sigma^2/(4  \alpha_{\kopt}^2)$. Thus the method naturally finds a solution which has increasing smoothness with increasing $k$ and solutions obtained without truncation will exhibit larger error due to increased smoothing. An effective and practical algorithm that implements the theory has also been provided, and  validated for $2$D image deblurring. These results expand on recent research on the characterization of the regularization parameter as closely dependent on the size of the singular subspace represented in the solution, \cite{FenuGCV,Renaut2017,RVA:15}. As there is a resurgence of interest in using a TSVD solution for the solution of ill-posed problems due to increased feasibility of finding a good approximation of a dominant singular subspace  using techniques from randomization, e.g. \cite{RandNLA,FastLS,GowerandRichtarik,Mahoney,LSRN,Overdetermined}, the results are more broadly relevant for more efficient estimates of the TSVD.  Implementation of the algorithm in these contexts is a topic for future work.

\bibliographystyle{siam}

%\include{Manuscript.bbl}
%\bibliography{UPRE}

\end{document}